\documentclass[11pt,a4paper]{amsart}

\usepackage{amsmath}
\usepackage{amssymb}
\usepackage{graphicx}
\usepackage{color}
\usepackage{url}
\usepackage{float}
\usepackage{hyperref}
\hypersetup{
    colorlinks,
    citecolor=black,
    filecolor=black,
    linkcolor=black,
    urlcolor=black
}

\newtheorem{thm}{Theorem}
\newtheorem{cor}[thm]{Corollary}

\newtheorem{prop}[thm]{Proposition}
\newtheorem{lem}[thm]{Lemma}
\newtheorem{remark}[thm]{Remark}
\newtheorem{conj}[thm]{Conjecture}
\newtheorem{ex}[thm]{Example}

\numberwithin{thm}{section}

\newcommand{\R}{\mathbf{R}}

\newcommand{\Z}{\mathbf{Z}}
\newcommand{\N}{\mathbf{N}}
\newcommand{\Ss}{\mathbf{S}}
\newcommand{\simd}{\mathrel{\rotatebox[origin=c]{-90}{$\sim$}}}

\begin{document}
\title{Minimal Penner dilatations on nonorientable~surfaces}
\author{Livio Liechti}
\thanks{The first author was supported by the Swiss National Science Foundation~(\# 175260)}
\address{Department of Mathematics, University of Fribourg, Ch.\ du Mus\'ee, 1700 Fribourg, Switzerland}
\email{livio.liechti@unifr.ch}
\author{Bal\'azs Strenner}
\address{Georgia Institute of Technology, School of Mathematics, Atlanta GA
  30332, USA}
\email{strenner@math.gatech.edu}

\begin{abstract}
For any nonorientable closed surface, we determine the minimal dilatation among pseudo-Anosov mapping classes arising from Penner's construction.
We deduce that the sequence of minimal Penner dilatations has exactly two
accumulation points, in contrast to the case of orientable surfaces where there
is only one accumulation point.
One of our key techniques is representing pseudo-Anosov dilatations as roots of
Alexander polynomials of fibred links and comparing dilatations using
the skein relation for Alexander polynomials.
\end{abstract}
\maketitle

\section{Introduction}
Thurston's classification states that elements of the mapping class group of a
finite type surface come in three types: reducible, periodic and pseudo-Anosov~\cite{Th}.
Associated with each pseudo-Anosov mapping class is a number~$\lambda$, the \emph{dilatation} or \emph{stretch factor}, which has several characterisations and is an algebraic integer~\cite{Th}.
In this article, we study the minimal dilatation among pseudo-Anosov mapping classes arising from a construction by products of Dehn twists along suitable simple closed curves, due to Penner~\cite{Pe}.
Let~$N_g$ be the nonorientable closed surface of genus~$g$, that is, the connected sum of~$g$ copies of the real projective plane~$\mathbf{RP}^2$.

\begin{thm}
\label{dilatationlimits}
For the minimal dilatation~$\delta_P(N_g)$ among pseudo-Anosov homeomorphisms arising from Penner's construction on the nonorientable closed surface of genus~$g$,
the limits $\lim_{k\to\infty}\delta_P(N_{2k})$ and $\lim_{k\to\infty}\delta_P(N_{2k+1})$ exist, and
\begin{enumerate}
\item[(a)] $\lim_{k\to\infty}\delta_P(N_{2k}) = 3+2\sqrt{2}$,
\item[(b)] $\lim_{k\to\infty}\delta_P(N_{2k+1})>3+2\sqrt{2}$.
\end{enumerate}
\end{thm}

We prove Theorem~\ref{dilatationlimits} by finding, for each nonorientable closed surface, a pseudo-Anosov mapping class
which has minimal dilatation among all pseudo-Anosov mapping classes arising from Penner's construction on this surface.
When the genus is even, we give a concrete description of the minimal dilatations~$\delta_P(N_{2k})$.

\begin{thm}
\label{evengenusexact}
For~$k\ge2$, the minimal dilatation~$\delta_P(N_{2k})$ equals the largest real
solution of the equation~$t-t^{\frac{k}{2k-1}}-t^{\frac{k-1}{2k-1}}-1=0.$
Alternatively, $\delta_P(N_{2k})$ equals the $2k-1$st power of the largest real root
of the integral polynomial~$x^{2k-1}-x^{k}-x^{k-1}-1$.
\end{thm}

Such a description is possible due to a rotational symmetry of the mapping classes
realising~$\delta_P(N_{2k})$. Unfortunately, the mapping classes
realising~$\delta_P(N_{2k+1})$ do not have such a rotational symmetry, so we do not obtain
such a concrete description of the minimal dilatations~$\delta_P(N_{2k+1})$.
However, we do have a description of~$\delta_P(N_{2k+1})$ as the largest
eigenvalue of a certain product of matrices, so using a computer we can
compute~$\delta_P(N_{2k+1})$ for at least up to $k=100$. These computations
strongly suggest the following.

\begin{conj}
  The limit $\lim_{k\to\infty}\delta_P(N_{2k+1})$ is the largest real root of
  the polynomial $x^4-8x^3+13x^2-8x+1$, approximately~$6.071360241468951$.
\end{conj}

For example, $\delta_P(N_{101})$ approximates the conjectured limit until 9 decimal
places and $\delta_P(N_{151})$ approximates it until at least 15 decimal places.

We find it intriguing that the coefficients of this polynomial are Fibonacci numbers, considering that the golden ratio makes a frequent appearance in the literature on minimal dilatations
(for example, see Theorem 1.11 and Question 1.12 in \cite{Hironaka}).

\subsection*{Motivation}
A \emph{mapping class} is a homeomorphism of a surface of finite type, up to isotopy keeping the boundary fixed pointwise.
A mapping class is \emph{pseudo-Anosov} if it has a representative~$f$ for which there exists a pair of transverse singular measured foliations~$\mathcal{F}^u$ and~$\mathcal{F}^s$ such that
~$f(\mathcal{F}^u)=\lambda\mathcal{F}^u$ and~$f(\mathcal{F}^s)=\lambda^{-1}\mathcal{F}^s$.

The dilatation of a pseudo-Anosov mapping class is a measure of its complexity, and one can ask about the minimal dilatation of a pseudo-Anosov
mapping class on a given surface. On orientable closed surfaces, the exact value of the minimal dilatation is known only for genus~$g=1$ and~$2$,
where the case~$g=2$ is due to Cho and Ham~\cite{ChoHam}.
Even when restricting to pseudo-Anosov maps with orientable invariant foliations, the minimal dilatations are known only for~$g=1,2,3,4,5,7$ and~$8$,
see the work of Lanneau and Thiffeault~\cite{LanneauThiffeult}.
Also for closed nonorientable surfaces, the minimal dilatation is known only in finitely many cases, by previous results of the authors~\cite{LiSt}.

A more approachable problem would be to obtain an accurate asymptotic behaviour
for the minimal dilatation as the genus~$g$ of the surface goes to infinity.
In~\cite{Pe2}, Penner showed that the minimal dilatation to the power~$g$ is bounded between positive constants.
However, it is not known if the limit of this normalised
dilatation exists. For some speculation and questions, see the work of Hironaka~\cite{Hironaka} (for orientable surfaces) and the authors~\cite{LiSt} (for nonorientable surfaces).

In this article, we study the dilatation of pseudo-Anosov mapping classes arising from a construction by products of Dehn twists along suitable simple closed curves, due to Penner~\cite{Pe}.
This construction might be representative of general phenomena concerning the minimal dilatation question for two reasons.
Firstly, Penner used roots of pseudo-Anosov mapping classes arising from his construction to obtain examples whose dilatations have minimal asymptotics~\cite{Pe2}.
Secondly, on many nonorientable surfaces, the authors find the minimal dilatation among pseudo-Anosov mapping classes with an orientable invariant foliation
by taking roots of pseudo-Anosov mapping classes arising from Penner's construction~\cite{LiSt}.
In this light, Theorem~\ref{dilatationlimits} can be seen as pointing towards the possibility that for nonorientable closed surfaces,
the sequence of minimal normalised dilatations among pseudo-Anosov
mapping classes with an orientable invariant foliation does not converge.

\subsection*{The minimising examples}
Figure~\ref{nonorientable_examples} depicts the genus six and seven nonorientable closed surfaces (with one open disc removed) as a surface obtained by glueing together five or six twisted annuli, respectively.
The minimising examples are obtained by applying a Dehn twist along the core curve of each of those annuli.
The order of the twisting should be, in a sense we make precise later, as bipartite as possible.
We find that for each nonorientable closed surface, the mapping classes of this kind minimise the dilatation among
pseudo-Anosov mapping classes arising from Penner's construction,
where the size of the cycle of annuli glued together is determined by the genus of the surface (Theorems~\ref{evengenus} and~\ref{oddgenus}).

\begin{figure}[h]
\begin{center}
\def\svgwidth{142pt}
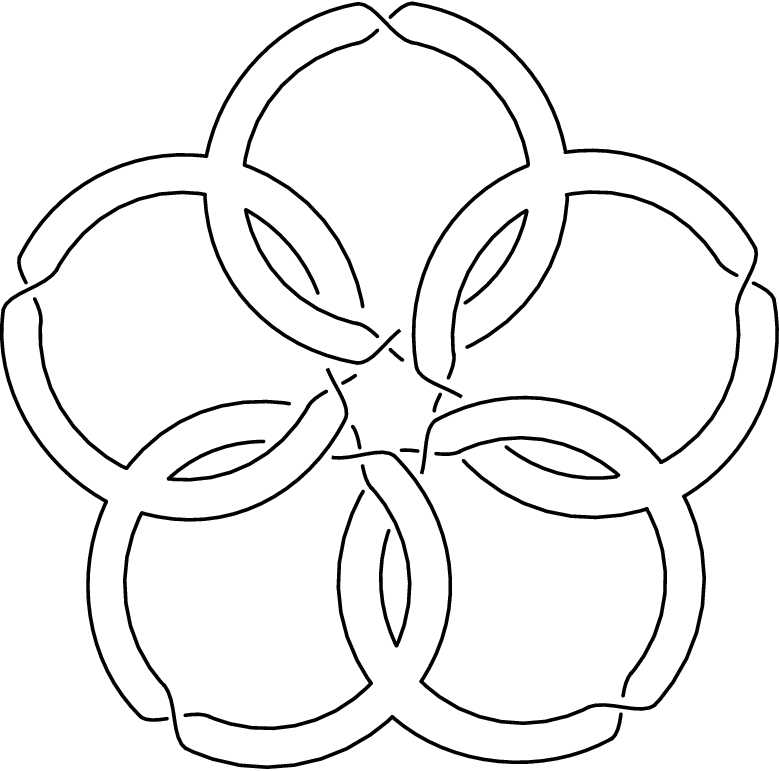
\def\svgwidth{183pt}
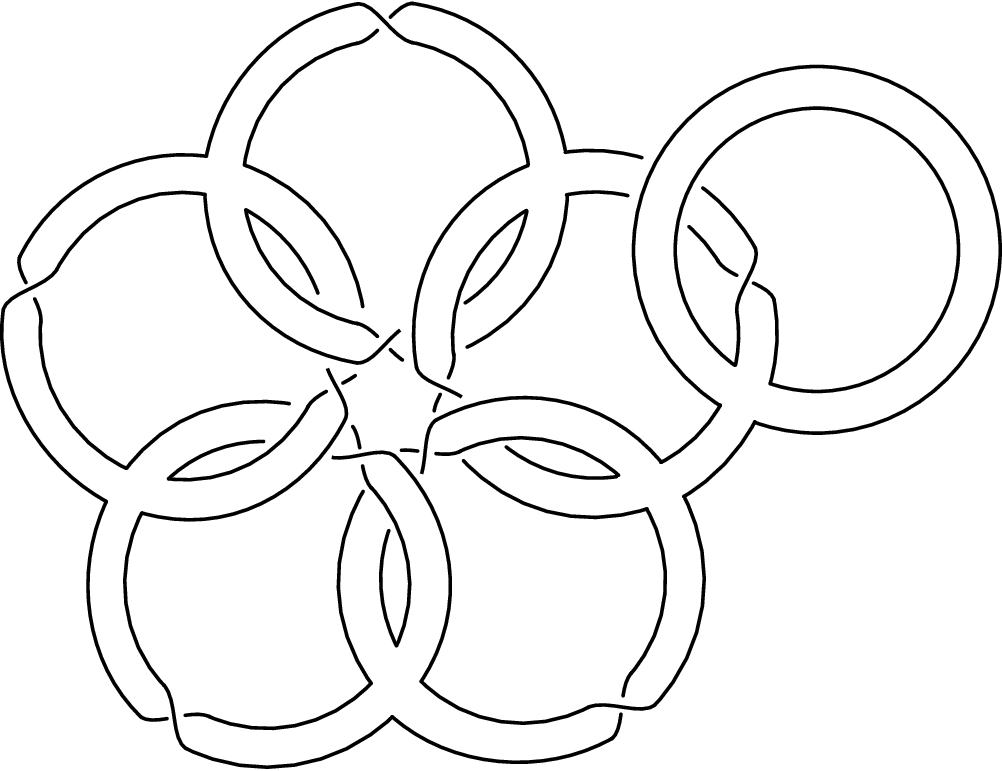
\caption{A connected sum of six (on the left) and seven (on the right) copies of $\mathbf{RP}^2$ minus a disc.}
\label{nonorientable_examples}
\end{center}
\end{figure}

\subsection*{Orientable versus nonorientable surfaces}
A remarkable difference between the case of orientable and nonorientable surfaces is that in the case of orientable closed surfaces~$S_g$,
the limit of the sequence of minimal dilatations~$\delta_P(S_g)$ arising from Penner's construction
exists, whereas in the nonorientable case, the limit of the sequence~$\delta_P(N_g)$ does not.
An orientable double cover argument implies that any accumulation point of~$\delta_P(N_g)$ must be at least~$3+2\sqrt{2}$, which is the limit in the orientable case
by work of the first author~\cite{Li}.
In fact, the limit in the orientable case is the same as for the even genus subsequence in the nonorientable case.

\subsection*{Fibered link techniques for odd genus}
The rotational symmetry of the even genus examples (Figure~\ref{nonorientable_examples}, left)
helps us to determine their dilatations. In order to deal with nonorientable surfaces of odd genus, where such a strong symmetry is lacking,
we use the theory of fibred links in~$\Ss^3$. More precisely, we describe the dilatation of these examples as the largest real root of the Alexander polynomial
of some fibred link in~$\Ss^3$, obtained by plumbing positive and negative Hopf bands to a disc.
In order to distinguish between different product orders of Dehn twists, we calculate the difference of the associated
Alexander polynomials, by using the skein relation for the Alexander polynomial.
We find that this difference is of a very particular form, which allows us to deduce the monotonicity properties needed to single out the examples with minimal dilatation.\\

\noindent
\textbf{Organisation.}
In Section~\ref{Penner_section}, we recall Penner's construction of pseudo-Anosov mapping classes and discuss the particularities of the construction for nonorientable surfaces.
In Sections~\ref{cycle_section} and~\ref{enriched_section}, we develop the theory of dilatations for pseudo-Anosov mapping classes which arise from Penner's construction using
simple closed curves with a cycle or an enriched cycle, respectively, as their intersection graph.
In Sections~\ref{evengenus_section} and~\ref{oddgenus_section}, we single out the dilatation-minimising examples among mapping classes arising from Penner's construction on nonorientable closed surfaces.
We are able to prove Theorem~\ref{evengenusexact} and all statements of Theorem~\ref{dilatationlimits} at the end of Section~\ref{evengenus_section}, except for the existence of the limit~$\lim_{k\to\infty}\delta_P(N_{2k+1})$.
This existence is proved at the end of Section~\ref{oddgenus_section}.
\\

\noindent
\textbf{Acknowledgements.}
We would like to thank Julien March\a'{e} for a fruitful discussion.
This project started during the Moduli Spaces workshop in Ventotene, and we are grateful to have been able to participate.
We also thank Dan Margalit and an anonymous referee for helpful remarks on an earlier version of this article.

\section{Penner's construction}
\label{Penner_section}
In this section, we describe Penner's construction for closed surfaces which need not be orientable.

We assume all mentioned two-sided simple closed curves~$c$ to be equipped with a homeomorphism~$\varphi_{c}$ from a regular neighbourhood~$U_c$ of~$c$ to the standard annulus~$A=\mathbf{S}^1\times[0,1]$.
The \emph{Dehn twist}~$T_c$ along~$c$ is then defined to be the identity outside~$U_c$ and~$\varphi_c^{-1}\circ T_A\circ\varphi_c$ inside~$U_c$, where~$T_A$ is the standard right Dehn twist of the annulus~$A$,
sending an arc that crosses the core curve of~$A$ to an arc that crosses the core curve but also winds around it once in the positive direction. In the case of an oriented surface,
the Dehn twist~$T_c$ along a curve~$c$ is positive or negative
if the homeomorphism~$\varphi_c$ is orientation-preserving or orientation-reversing, respectively.

In Penner's construction for orientable surfaces, we ask that if two curves intersect, one should be twisted along positively and the other should be twisted along negatively.
The notion of a positive or a negative Dehn twist does not make sense on a nonorientable surface, but one can still ask that locally at any intersection point,
the twisting should go in different directions: we say two curves~$c_1$ and~$c_2$ intersect \emph{inconsistently} if for every point~$p\in c_1\cap c_2$ the pullbacks of the orientation of~$A$ by~$\varphi_{c_1}$ and~$\varphi_{c_2}$ disagree.

\begin{thm}[Penner's construction]
\label{pennerconstruction_thm}
Let~$\{c_i\}$ be a collection of at least two two-sided curves which intersect inconsistently and without bigons, and whose union fills a closed surface~$\Sigma$.
Let~$\mathcal{P}$ be the monoid generated by the Dehn twists~$T_{c_i}$.
Define~$\rho: \mathcal{P} \to \mathrm{SL}(n, \Z)$ by
\begin{align*}
\rho(T_{c_i}) = I + R_{c_i},
\end{align*}
and extend linearly,
where the matrices~$R_{c_i}$ are obtained from the
geometric intersection matrix~$\Omega$
of the curves~$\{c_1,\dots,c_n\}$
by setting all entries to zero which are not in the row corresponding to~$c_i$.
Then each~$\phi\in\mathcal{P}$ such that every~$c_i$ gets twisted along at least once is pseudo-Anosov and its dilatation equals the Perron-Frobenius eigenvalue of~$\rho(\phi)$.
\end{thm}

For more details and proofs, see Penner's original article~\cite{Pe} or Fathi's alternative approach~\cite{Fa}.
Penner's construction for nonorientable surface is also explained in more detail by the second author~\cite{Str}.

We call the dilatation of a mapping class arising from Penner's construction a \emph{Penner dilatation}.
In fact, a Penner dilatation only depends on the intersection graph
of the collection~$\{c_i\}$ of curves used and the product order of the twists.
Here, the \emph{intersection graph} has one vertex for each curve~$c_i$ and two vertices are connected by an edge of multiplicity~$k$ if and only if their corresponding curves intersect~$k$ times.

\begin{ex}
\label{cycle_example}
\emph{
Take an odd number~$l$ of annuli and glue them together to form a circle.
Insert two half-twists in each annulus in order to produce a nonorientable surface. This is depicted in Figure~\ref{nonorientable_examples} on the left for~$l=5$.
Finally, glue in a disc along the boundary component (there is only one boundary component)
to obtain a closed nonorientable surface, which by a direct Euler characteristic count is shown to be of genus~$g=l+1$.
Number the core curves~$c_i$ of the annuli in the clockwise fashion. It is not hard to see that one can find homeomorphisms~$\varphi_{c_i}$ from regular neighbourhoods of the curves~$c_i$
to the standard annulus so that the curves~$c_i$ intersect inconsistently. Furthermore, there are no bigons, since each pair of curves~$c_i$ and~$c_j$ intersects at most once.
It follows that the collection of core curves~$c_i$ satisfies the hypotheses of Penner's construction.
The intersection graph is a cycle of length~$l$.
}\end{ex}

\begin{ex}\emph{
The cycles from Example~\ref{cycle_example} can be modified. For example, we can glue an extra band to one of the~$l$ bands forming the cycle.
This is depicted in Figure~\ref{nonorientable_examples} on the right for~$l=5$.
As in Example~\ref{cycle_example}, one can see that there is still only one boundary component, along which we glue in a disc to obtain
a closed nonorientable surface of genus~$g=l+2$.
It is directly checked that the core curves of these examples also satisfy the hypotheses of Penner's construction.
The intersection graph is a cycle of length~$l$ with an extra vertex added, a graph which we call an \emph{enriched cycle}.
}\end{ex}

\subsection{Nonorientable surfaces}We now give three simple observations concerning the nonorientable case of Penner's construction,
which will be used later in the paper. Lemma~\ref{notbipartite} hints at why
searching for the minimal dilatation among pseudo-Anosov mapping classes arising from Penner's construction is more complicated
on nonorientable surfaces than on orientable ones: the intersection graph of the curves used in the construction always contains at least one cycle,
while for the minimising examples on closed orientable surfaces, it is a path~\cite{Li}.

\begin{lem}
\label{notbipartite}
If a collection of curves $\{c_i\}$ as in Penner's construction fills a nonorientable surface, then their intersection graph is not bipartite.
\end{lem}

\begin{proof}
Let $\Sigma$ be a nonorientable closed surface, and let~$\{c_i\}$ be a collection of curves as in Penner's construction that fill~$\Sigma$.
Recall that there exist homeomorphisms~$\varphi_{c_i}$ of regular neighbourhoods of the curves $c_i$ to the standard annulus such that at each intersection point the pullback orientations disagree.
If the intersection graph of the curves~$\{c_i\}$ were bipartite, we could simply switch the orientation
of the regular neighbourhoods of the the curves corresponding to one set of the bipartition to obtain a situation
in which at each intersection point, the orientations of the regular neighbourhoods agree. In particular, as the curves~$\{c_i\}$ are assumed to fill the surface~$\Sigma$,
we could extend this consistent orientation to an orientation of the surface~$\Sigma$.
\end{proof}

\begin{lem}
\label{fillingcycle}
Let~$\Sigma$ be any surface, and
let~$\{c_1,\dots, c_l\}$ be a collection of two-sided curves in~$\Sigma$ that intersect inconsistently and with the pattern of a cycle of odd length~$l$.
Then, a small regular neighbourhood~$\Sigma_0$ of the
union of the curves~$c_i$ is homeomorphic to~$N_{l+1}$ minus a disc.
\end{lem}

In particular, a collection of two-sided curves $\{c_i\}$ that intersect inconsistently and with the pattern of an odd cycle can only fill a nonorientable closed surface of even genus.
Indeed, applying Lemma~\ref{fillingcycle} to a collection of curves~$\{c_1,\dots, c_l\}$ that in addition fill a closed surface~$\Sigma$, we directly obtain the following statement.

\begin{cor}
\label{fillingcyclecor}
Let~$\Sigma$ be a closed surface, and
let~$\{c_1,\dots, c_l\}$ be a collection of two-sided curves in~$\Sigma$ that intersect inconsistently and with the pattern of a cycle of odd length~$l$.
If the collection of curves~$\{c_1,\dots, c_l\}$ fills~$\Sigma$, then~$\Sigma$ is homeomorphic to~$N_{l+1}$.
\end{cor}

\begin{proof}[Proof of Lemma~\ref{fillingcycle}]
 Let~$\{c_1,\dots, c_{l}\}$ be a collection of two-sided curves that intersect
 inconsistently and with the pattern of a cycle of odd length~$l$. We want to show
 that the boundary of a small regular neighbourhood~$\Sigma_0$ of the union of
 the curves~$c_i$ has exactly one boundary component.
 The statement then follows directly from the fact that~$\Sigma_0$ is homotopy equivalent to a wedge of~$l+1$ circles, and hence has Euler characteristic~$-l$.

Consider the surface~$\Sigma_1$ obtained from~$\Sigma_0$ by removing a \emph{square}: the intersection of the annulus neighbourhoods of~$c_1$ and~$c_{l}$. The surface~$\Sigma_1$ is homeomorphic to the surface obtained by chaining together~$l$ annuli and removing a square from the first and last annuli as in Figure~\ref{cutchain}.
\begin{figure}[h]
\begin{center}
\def\svgwidth{370pt}
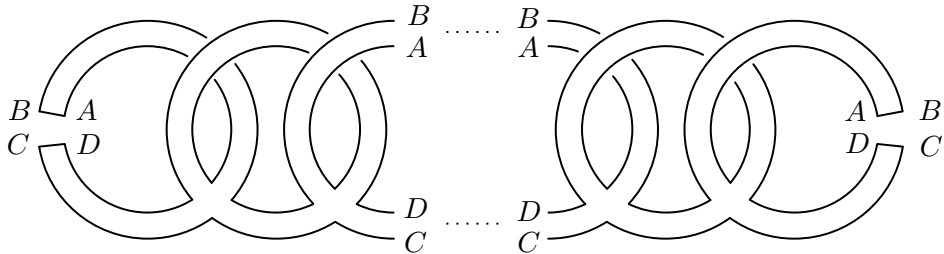
\caption{The surface~$\Sigma_1$. The letters~$A,B,C$ and~$D$ indicate how the strands of the boundary~$\partial\Sigma_1$ connect.}
\label{cutchain}
\end{center}
\end{figure}
The boundary of the first and last annuli each have four arcs on the boundary of~$\Sigma_1$. In~$\partial \Sigma_0 \cap \partial \Sigma_1$, the four arcs on the first annulus are connected to the four arcs on the last annulus as shown on Figure~\ref{cutchain}.

To reverse the process and construct the surface~$\Sigma_0$ from~$\Sigma_1$, we need to glue~$\partial \Sigma_1 \setminus \partial \Sigma_0$ to a square. Since the curves~$c_i$ are assumed to intersect inconsistently, there are two ways to do this, see Figure~\ref{glueingsquare}.
\begin{figure}[h]
\begin{center}
\def\svgwidth{390pt}
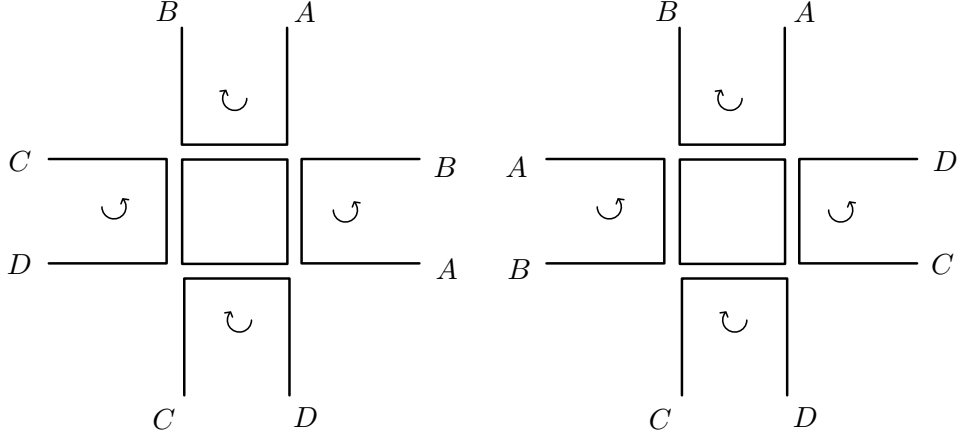
\caption{Glueing together the leftmost annulus (depicted vertically) and the rightmost annulus (depicted horizontally) from Figure~\ref{cutchain} so that the orientations do not agree on the intersection.}
\label{glueingsquare}
\end{center}
\end{figure}
We can see that in each case, all arcs get identified to a single boundary component.
\end{proof}

A subgraph~$\Gamma'$ of the intersection graph~$\Gamma$ is \emph{induced} if it contains all edges of~$\Gamma$ that connect pairs of vertices of~$\Gamma'$.

\begin{lem}
\label{genusbound}
Let~$\Sigma$ be a surface filled by a collection of curves~$\{c_i\}$ satisfying the hypotheses of Penner's construction.
If the intersection graph~$\Gamma$ contains a cycle of odd length~$l$ as an induced subgraph, then~$\Sigma$ is nonorientable and its genus is greater than or equal to~$l+1$.
\end{lem}

\begin{proof}
Let~$c_{i_1},\dots ,c_{i_l}\subset\Sigma$ be the curves corresponding to the induced cycle of length~$l$.
This means that two consecutive curves~$c_{i_j}$ and~$c_{i_{j+1}}$ intersect once, where the index~$j$ is taken~$(\mathrm{mod}\ l)$.
Since the cycle of length~$l$ is an induced subgraph of~$\Gamma$, there are no other intersections between curves~$c_{i_j}$.
By Lemma~\ref{fillingcycle}, a regular neighbourhood~$U$ of the union of the curves~$c_{i_1},\dots ,c_{i_l}$ is not orientable, and hence neither is~$\Sigma$.
Furthermore,~$U$ has exactly one boundary component, and~$\chi(U)=-l$.
Thus, the surface~$\Sigma$ has the nonorientable closed surface of genus~$l+1$ as a connected summand, and, in particular, is of genus at least~$l+1$ itself.
\end{proof}

\section{Dilatation theory of the cycle}
\label{cycle_section}
The goal of this section is to describe the dilatations arising from Penner's construction using curves with an odd cycle as their intersection graph,
such as in Example~\ref{cycle_example}.

Let~$C_l$ be a cycle of length~$l$ encoding the intersection of curves used in Penner's construction: to each curve~$c_i$ corresponds a vertex~$v_i$ of~$C_l$.
We now study mapping classes defined by a word~$w$ in the Dehn twists~$T_{c_i}$ so that every twist~$T_{c_i}$ appears exactly once.
To every such word, we associate an acyclic orientation of~$C_l$: an edge between~$v_i$ and~$v_j$ is directed from~$v_i$ to~$v_j$ if~$T_{c_i}$ occurs in~$w$ before~$T_{c_j}$ and vice-versa.
The \emph{flow difference} of an acyclic orientation of the cycle~$C_l$ is the number of edges oriented in the clockwise sense minus the number of edges oriented in the anticlockwise sense.

\begin{lem}
\label{flowdiffconjugacy}
Let~$w$ and~$w'$ be two words in the Dehn twists~$T_{c_i}$ so that every twist appears exactly once in each of them.
If~$w$ and~$w'$ induce acyclic orientations of~$C_l$ with the same flow difference, then the matrices~$\rho(w)$ and~$\rho(w')$ from Penner's construction are conjugate.
\end{lem}

\begin{proof}
Let~$W$ be the set of words in the Dehn twists~$T_{c_i}$ so that every twist appears exactly once.
By a result of Shi~\cite{Shi}, there exists a one-to-one correspondence between acyclic orientations of the cycle and words in~$W$ up to the commutation relation of Dehn twists
(which commute exactly if the defining curves do not intersect).
Moreover, two acyclic orientations of the cycle are connected by a sequence of source-to-sink operations if and only if they have the same flow difference by a result of Pretzel~\cite{Pretzel}.
Here, a \emph{source-to-sink operation} denotes the process of making a source of the directed graph into a sink by switching the orientations of all adjacent edges.
By Shi's correspondence, on the level of the words, making a source into a sink or vice-versa translates to a conjugation by the Dehn twist along the corresponding curve.
In particular, the two matrices~$\rho(w)$ and~$\rho(w')$ associated with two pseudo-Anosov mapping classes arising from Penner's construction
are conjugate if the words~$w$ and~$w'$ induce acyclic orientations of~$C_l$ with the same flow difference.
\end{proof}

Since conjugate matrices have the same eigenvalues, we only have to study one standard representative for each flow difference.
By symmetry, we also have to consider only the absolute value of the flow difference.

\subsection{A formula for the dilatation}
The goal of this section is to show that for a cycle of fixed length, the dilatation of Dehn twist products is a strictly increasing function of the absolute value of the flow difference.
This follows from the Propositions~\ref{function} and~\ref{specialisation} below, which give the means to directly compute the dilatation given the length of the cycle and the flow difference.

Let~$C=\{(x,y)\in\R^2:y>0, |x|<y\}$. Furthermore, define the function~$f:C\to \R_{>0}$ by mapping~$(x,y)$ to the largest real solution of the equation~$t-t^{\frac{y+x}{2y}}-t^{\frac{y-x}{2y}}-1=0$.

\begin{prop}
\label{function}
The function~$f$ is well-defined and
\begin{enumerate}
\item is~$0$-homogeneous, and, in fact, only depends on~$|\frac{x}{y}|$,
\item is continuous and strictly increasing in~$|\frac{x}{y}|$.
\end{enumerate}
\end{prop}

\begin{prop}
\label{specialisation}
For a tuple~$(d,l)\in\Z^2\cap C$ such that~$d\equiv l(\mathrm{mod}\ 2)$, the value~$f(d,l)$ equals the dilatation of the Penner mapping classes with flow difference~$d$ on the cycle of length~$l$.
\end{prop}

\begin{proof}[Proof of Proposition~\ref{function}]
Notice that~$$t-t^{\frac{y+x}{2y}}-t^{\frac{y-x}{2y}}-1=t-t^{\frac{1}{2}+\frac{x}{y}}-t^{\frac{1}{2}-\frac{x}{y}}-1=t-t^{\frac{1}{2}+|\frac{x}{y}|}-t^{\frac{1}{2}-|\frac{x}{y}|}-1.$$
This proves~(1), assuming that~$f$ is well-defined. Define~$$h(t,s)=t-t^{\frac{1}{2}+s}-t^{\frac{1}{2}-s}-1$$ for~$0<s<\frac{1}{2}$.
For every~$s$,~$h(1,s)=-2$. Furthermore,~$\partial_th(t,s)>0$ for all~$t>1$. It follows that for any fixed~$s$, the function~$h(\cdot,s)$ has exactly one real zero~$>1$.
This shows that~$f$ is well-defined. Furthermore, $\partial_th(t,s)$ depends continuously on~$s$, therefore so does the real zero~$>1$ of the function~$h(\cdot,s)$.
This proves the first part of~(2). In order to see the second part of~(2), notice that~$\partial_sh(t,s)<0$.
This implies that the real zero~$>1$ of the function~$h(\cdot,s)$ is strictly increasing in~$s$.
\end{proof}

\begin{ex}[Twist and click homeomorphisms]
\label{twistandclick}
\emph{
Let~$l,c\in\N$ be natural numbers such that~$c<l$ and~$\mathrm{gcd}(c,l)=1$.
Let~$\Sigma_l$ be the surface obtained by thickening a collection of~$l$ curves with the~$l$-cycle as their intersection graph, so that between any two intersections there is a half-twist.
This is depicted for~$l=5$ in Figure~\ref{nonorientable_examples} on the left.
Consider the mapping class~$\phi_{l,c}$ obtained by a Dehn twist along one of the curves composed with a~$c$-fold click, that is,
a rotation of the the symmetric surface~$\Sigma_l$ by an angle~$c\cdot\frac{2\pi}{l}$.
The~$l$-th power of such a mapping class~$\phi_{l,c}$ arises from Penner's construction using the core curves of the annuli with the~$l$-cycle as their intersection graph,
and every curve gets twisted along exactly once.
Since~$\phi_{l,c}^l$ is pseudo-Anosov by Penner's construction, so is~$\phi_{l,c}$ by the classification of surface homeomorphisms and the dilatation of~$\phi_{l,c}$ is the~$l$-th
root of the dilatation of~$\phi_{l,c}^l$.
}\end{ex}

The following lemma describes the dilatation of the twist and click mapping classes introduced in Example~\ref{twistandclick}.
For~$c=2$, the result is also stated by the authors in~\cite{LiSt}. In this case the absolute value of the flow difference is~$1$.
The proofs are basically identical.

\begin{lem}
\label{companion}
Let~$a$ be the smallest natural number such that~$ac\equiv 1(\mathrm{mod}\ l)$.
Then, the dilatation of~$\phi_{l,c}$ is given by the largest real root of the polynomial~$t^l -t^{l-a}-t^{a}-1$.
\end{lem}

\begin{proof}
The mapping class~$\phi_{l,c}^l$ is pseudo-Anosov and arises from Penner's construction.
Furthermore, the associated matrix~$\rho(\phi_{l,c}^l)$ in Penner's construction equals the action on the first homology of the surface induced by~$\phi_{l,c}^l$.
To see this, choose the collection of the core curves of the annuli as a basis for the first homology, oriented invariantly under rotation.
From this it follows that~$\phi_{l,c}^l$ has an orientable invariant foliation, and hence so does~$\phi_{l,c}$.
In particular, the dilatation of~$\phi_{l,c}$ is given by the spectral radius of its action induced on the first homology of the surface,
which we describe now.
Number the core curves of the twisted bands in the following way. The first curve~$c_1$ is the one along which we do a Dehn twist in the definition of~$\phi_{l,c}$.
The second curve~$c_2$ is the image of~$c_1$ under rotation of~$\Sigma_l$ by an angle~$-c\cdot\frac{2\pi}{l}$.
The third curve~$c_3$ is the image of~$c_2$ under rotation of~$\Sigma_l$ by an angle~$-c\cdot\frac{2\pi}{l}$, and so on.
As a basis for the first homology~$\mathrm{H}_1(\Sigma_l;\R)$, we choose the homology classes of~$c_1, c_2, \dots, c_{l}$.
We obtain that the rotation~$r$ of~$\Sigma_l$ by an angle~$c\cdot\frac{2\pi}{l}$ acts by a permutation matrix, sending~$c_i$ to~$c_{i-1}$, where the indices are taken~$(\mathrm{mod}\ l)$.
Furthermore, the Dehn twist~$T_{c_1}$ acts as the identity on the curves~$c_i$ for~$i\ne a,l-a$,
and adds the curve~$c_1$ to the curves~$c_a$ and~$c_{l-a}$.
The product of these matrix actions is a companion matrix for the polynomial~$t^l -t^{l-a}-t^{a}-1$.
For example, for~$l=5$ and~$c=1$, we have~$a=1$ and
\begin{align*}
(T_{c_1})_\ast = \begin{pmatrix}
1 & 1 & 0 & 0 & 1\\
0 & 1 & 0 & 0 & 0\\
0 & 0 & 1 & 0 & 0\\
0 & 0 & 0 & 1 & 0\\
0 & 0 & 0 & 0 & 1
\end{pmatrix},
r_\ast &= \begin{pmatrix}
0 & 1 & 0 & 0 & 0\\
0 & 0 & 1 & 0 & 0\\
0 & 0 & 0 & 1 & 0\\
0 & 0 & 0 & 0 & 1\\
1 & 0 & 0 & 0 & 0
\end{pmatrix},\\
(\phi_{5,1})_\ast = r_\ast \cdot(T_{c_1})_\ast &= \begin{pmatrix}
0 & 1 & 0 & 0 & 0\\
0 & 0 & 1 & 0 & 0\\
0 & 0 & 0 & 1 & 0\\
0 & 0 & 0 & 0 & 1\\
1 & 1 & 0 & 0 & 1
\end{pmatrix},
\end{align*}
so~$(\phi_{5,1})_\ast$ has characteristic polynomial~$t^5-t^4-t-1$.
\end{proof}

\begin{lem}
\label{flowdiff}
Let~$a$ be the smallest natural number such that~$ac\equiv 1(\mathrm{mod}\ l)$. Then, the flow difference of~$\phi_{l,c}^l$ is~$l-2a$.
\end{lem}

\begin{proof}
We identify the elements of~$\Z/l\Z$ with the vertices of the cycle~$C_l$ of length~$l$.
We consider the sequence of residues~$0,c,2c,\dots, (l-1)c \in\Z/l\Z$, which is the sequence in which~$\phi_{l,c}^l$ twists along the curves (corresponding to elements of~$\Z/l\Z$).
In order to determine the flow difference of~$\phi_{l,c}^l$, it suffices to know for each pair of adjacent elements~$k,k+1\in\Z/l\Z$ which element appears first in the sequence.
Indeed, if~$k\in\Z/l\Z$ appears first in the sequence, then the edge connecting the~$k$th and the~$k+1$st vertex is oriented towards the~$k+1$st vertex, and vice versa.

Assume for a moment that~$a$ is minimal so that~$ac\equiv \pm1(\mathrm{mod}\ l)$.
Then, the residue classes~$0,c, 2c, \dots, (a-1)c\in\Z/l\Z$ are pairwise nonadjacent.
Now, the next residue class in the sequence is~$ac\equiv1\equiv0+1(\mathrm{mod}\ l)$. Each element that occurs in the sequence after~$ac$ can also be obtained by adding~$1$ to an element that occurred
already before in the sequence. We deduce that we obtain~$l-a$ edges pointing in the clockwise direction and~$a$ edges pointing in the anticlockwise direction.
This yields a flow difference of~$(l-a)-a=l-2a$.

If~$a$ is not the minimal natural number so that~$ac\equiv \pm1(\mathrm{mod}\ l)$, then we have~$l-a<a$ and~$(l-a)c\equiv -1(\mathrm{mod}\ l)$.
We can repeat the same argument, but the direction of each edge is switched. We obtain~$l-(l-a)$ edges pointing in the anticlockwise direction and~$l-a$ edges pointing in the clockwise direction.
This yields a flow difference of~$l-a-(l-(l-a))=l-2a$.
\end{proof}

\begin{proof}[Proof of Proposition~\ref{specialisation}]

Assume for a moment that~$l$ is odd. We first reduce to the case~$\mathrm{gcd}(d,l)=1$. For this, assume for a moment $\mathrm{gcd}(d,l)>1$. We have
$$f(d,l)=f(\frac{d}{\mathrm{gcd}(d,l)}, \frac{l}{\mathrm{gcd}(d,l)})$$ by~$0$-homogeneity of~$f$.
Note that a Penner mapping class of flow difference~$\frac{d}{\mathrm{gcd}(d,l)}$ on the cycle of length~$\frac{l}{\mathrm{gcd}(d,l)}$ is
covered~$\mathrm{gcd}(d,l)$-fold by a Penner mapping class with flow difference~$d$ on the cycle of length~$l$.
It therefore suffices to prove the statement for~$\mathrm{gcd}(d,l)=1$.

In the twist and click mapping classes for a fixed odd length~$l$, as~$c$ runs through the numbers smaller than~$l$ with~$\mathrm{gcd}(c,l)=1$,
also the corresponding~$a$ runs through the numbers smaller than~$l$ with~$\mathrm{gcd}(a,l)=1$. Therefore, the numbers~$l-2a$ run through the odd numbers of absolute value smaller than~$l$
with~$\mathrm{gcd}(l-2a,l)=1$. In particular, we obtain every flow difference~$d$ with~$\mathrm{gcd}(d,l)=1$ as an~$l$-th power of a twist and click example.
By Lemma~\ref{companion} and~\ref{flowdiff}, the dilatation of the Penner mapping classes with flow difference~$d$ on the cycle of length~$l$ is the~$l$-th power of the largest real root of the polynomial
~$t^l -t^{l-a}-t^{a}-1$, where~$a=(l-d)/2$. Equivalently, the dilatation equals the largest real solution of the equation
$$t-t^{\frac{l+d}{2l}}-t^{\frac{l-d}{2l}}-1=0,$$
which finishes the proof in the case where~$l$ is odd.

Now, let~$l$ be even. The proof of this case is similar, the main difficulty being that by dividing both~$l$ and~$d$ by~$\mathrm{gcd}(d,l)$, it is possible to break the condition~$d\equiv l(\mathrm{mod}\ 2)$.
This time, we reduce our argument to the case~$\mathrm{gcd}(d,l)=2$ and~$d\not\equiv l(\mathrm{mod}\ 4)$.
Indeed, this is exactly the case where the condition~$d\equiv l(\mathrm{mod}\ 2)$ does not hold anymore after
dividing both~$l$ and~$d$ by~$2$. Notice that any other case either reduces to this one or a case where~$l$ is odd, by a covering argument as above.
It therefore suffices to prove the statement for~$\mathrm{gcd}(d,l)=2$ and~$d\not\equiv l(\mathrm{mod}\ 4)$.

As in the argument for odd~$l$, we again use the twist and click mapping classes from Example~\ref{twistandclick}.
The only difference is that in this case, the numbers~$l-2a$ run through the even numbers of absolute value smaller than~$l$
with~$\mathrm{gcd}(l-2a,l)=2$ and~$l-2a\not\equiv l(\mathrm{mod}\ 4)$. Indeed, since~$a$ is odd, we obtain~$l-2a\not\equiv l(\mathrm{mod}\ 4)$.
On the other hand, since we get all~$a$ with~$\mathrm{gcd}(a,l)=1$ by varying~$c$ with~$\mathrm{gcd}(c,l)=1$,
we obtain all flow differences~$d$ with~$\mathrm{gcd}(d,l)=2$ and~$d\not\equiv l(\mathrm{mod}\ 4)$ by an~$l$-th power of a twist and click example.
Here, we have again used~Lemma~\ref{flowdiff} to argue that~$d=l-2a$.
As in the case of odd~$l$, the statement follows from Lemma~\ref{companion}.
\end{proof}

\begin{remark}\emph{
For odd~$l$ and~$c=2$, the twist and click mapping class~$\phi_{l,c}$ conjecturally minimises the dilatation
among pseudo-Anosov mapping classes with an orientable invariant foliation on the nonorientable closed surface of genus~$l+1$.
This has been shown for even genus up to~$20$ by the authors~\cite{LiSt}.
Adding the mapping classes~$\phi_{l,c}$ for other~$c$ to the picture as in Proposition~\ref{function} exhibits a strong similarity with theory of the normalised dilatation
on a fibred face of the Thurston norm ball~\cite{Fried79, Fried82}. Indeed, we expect many of the mapping classes~$\phi_{l,c}$ to lie in a common fibred cone.
}\end{remark}

\section{Even genus minimal dilatations}
\label{evengenus_section}

The goal of this section is to single out the minimal dilatation examples among mapping classes arising from Penner's construction on a closed nonorientable surface of even genus.
We will often use the following lemma to obtain lower bounds for the dilatation of mapping classes arising from Penner's construction.
It was implicitly used already in the case of orientable surfaces by the first author~\cite{Li}.

\begin{lem}
\label{treebound}
Let~$\phi$ be a mapping class arising from Penner's construction using a collection of curves~$\{c_i\}$.
If the intersection graph of the curves~$\{c_i\}$ contains a tree~$\Gamma$ (possibly with multiple edges between two vertices) as a subgraph, then
$$\lambda(\phi)\ge\frac{2+\alpha^2+\sqrt{4\alpha^2+\alpha^4}}{2},$$
where~$\alpha$ is the largest eigenvalue of the adjacency matrix of~$\Gamma$.
\end{lem}

\begin{proof}
Let~$\phi$ be a mapping class arising from Penner's construction using a collection of curves~$\{c_i\}$,
and let the tree~$\Gamma$ be a subgraph of the intersection graph of the curves~$\{c_i\}$.
For two matrices~$A$ and~$B$ of the same dimensions, we write~$A\le B$ if~$a_{ij}\le b_{ij}$ for all~$i,j$.
Recall that the spectral radius of nonnegative matrices is monotonic under~``$\le$'', see, for example,~\cite{BrHa}.
We may therefore assume that~$\phi$ is a product of Dehn twists~$T_{c_i}$ so that every curve~$c_i$ gets twisted along exactly once.
Let~$\phi_\Gamma$ be the subproduct of Dehn twists~$T_{c_i}$ along exactly those curves~$c_i$ which correspond to the vertices of~$\Gamma$.
We have that~$\lambda(\phi)$ is an upper bound for the spectral radius of~$\rho(\phi_\Gamma)$.
The spectral radius of~$\rho(\phi_\Gamma)$ is in turn an upper bound for the Penner dilatation $\lambda(\Gamma)$ associated with the subgraph~$\Gamma$ and its induced order of twisting.
Note that by a result of Steinberg, the order of twisting does not change the conjugacy class, since~$\Gamma$ is a tree~\cite{Steinberg}.
It follows that~$\lambda(\Gamma)$ is independent of the Dehn twist product order on~$\Gamma$.
In particular, we may calculate~$\lambda(\Gamma)$ as the dilatation of a product of two multitwists, in which case Thurston's construction yields
$$\lambda(\Gamma)+\lambda(\Gamma)^{-1} -2 =\alpha^2,$$
where~$\alpha$ is the largest eigenvalue of the adjacency matrix of~$\Gamma$, see~\cite{Th}.
Solving this equation for~$\lambda(\Gamma)$ yields the result.
\end{proof}

Let~$\varphi_l$ be the mapping class defined by the~$l$th power of the twist-and-click mapping class~$\phi_{l,2}$, where~$l$ is an odd natural number.
By Lemma~\ref{flowdiff}, if~$c=2$, then~$a=\frac{l+1}{2}$ and the absolute value of the flow difference associated with~$\varphi_l$ equals~$|l-2a|=1$.
Both Lemma~\ref{fromabove} and Lemma~\ref{bestflow} follow readily from Propositions~\ref{function} and~\ref{specialisation}.

\begin{lem}
\label{bestflow}
Among pseudo-Anosov mapping classes arising from Penner's construction using curves with an odd~$l$-cycle as their intersection graph,
the mapping class~$\varphi_l$ has minimal dilatation.
\end{lem}

\begin{proof}
By Proposition~\ref{specialisation}, the dilatation of the Penner mapping classes with flow difference~$d$ on the cycle of length~$l$ equals~$f(d,l)$.
By Proposition~\ref{function}, the function~$f(d,l)$ is strictly increasing in~$|\frac{d}{l}|=\frac{|d|}{l}$. This means that for a cycle of fixed length~$l$, the
dilatation is a strictly increasing function of the absolute value of the flow difference. In particular, the dilatation is minimised for the minimal absolute value of
the flow difference, which for a cycle of odd length~$l$ is~$1$.
\end{proof}

\begin{lem}
\label{fromabove}
We have~$\lambda(\varphi_{l+2})<\lambda(\varphi_l)$.
\end{lem}

\begin{proof}
By Proposition~\ref{specialisation}, we have that the dilatation of~$\lambda(\varphi_{j})$ is~$f(1,j)$.
By Proposition~\ref{function}, the function~$f(1,j)$ is strictly increasing in~$|\frac{1}{j}|=\frac{1}{j}$ and hence strictly decreasing in~$j$.
\end{proof}

We are now ready to describe the Penner mapping classes of minimal dilatation on nonorientable closed surfaces of even genus.
Note that we only have to consider nonorientable surfaces of genus at least four since the mapping class group of the Klein bottle is finite and thus does not contain
pseudo-Anosov elements.

\begin{thm}
\label{evengenus}
The mapping class~$\varphi_l$ has the minimal dilatation among pseudo-Anosov mapping classes arising from Penner's construction for a nonorientable closed surface of even genus~$g=l+1$.
\end{thm}

\begin{proof}
Let $N_{l+1}$ be the nonorientable closed surface of even genus~$l+1$. We know that there exists the mapping class~$\varphi_l$ on~$N_{l+1}$, with dilatation~$\lambda(\varphi_l)$.
Furthermore, let~$\phi$ be any mapping class on~$N_{l+1}$ arising from Penner's construction.
As before, we are allowed to assume that every curve used for the construction of~$\phi$ gets twisted along exactly once.
We distinguish cases depending on the intersection graph of the curves used in the construction of~$\phi$.

\emph{Case 1: the intersection graph contains a double edge.} Let~$c_1$ and~$c_2$ be two curves that intersect at least twice.
Since a bipartite family of curves which intersect inconsistently cannot fill a nonorientable surface, there must be at least one other curve~$c_3$ intersecting either~$c_1$ or~$c_2$.
In particular, the intersection graph of the curves~$\{c_i\}$ contains the tree~$\Gamma$ with three vertices, one double edge and one simple edge as a subgraph,
depicted in Figure~\ref{9subgraph} on the left.
The adjacency matrix of this tree has maximal eigenvalue~$\sqrt{5}$ and we use Lemma~\ref{treebound} to conclude
$$\lambda(\phi)\ge\frac{7+3\sqrt{5}}{2}\approx6.854.$$
This number is larger than the dilatation of any mapping class~$\varphi_l$ by the values given in Table~\ref{cyclecasestable} and the monotonicity due to Lemma~\ref{fromabove}.

\begin{figure}[h]
\begin{center}
\def\svgwidth{150pt}
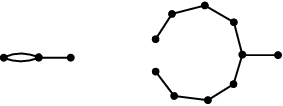
\caption{}
\label{9subgraph}
\end{center}
\end{figure}

\emph{Case 2: the intersection graph contains an odd cycle of length~$k\le l$}: In this case, by an argument similar to the argument used to prove Lemma~\ref{treebound},
the dilatation is always bounded from below by the dilatation of a pseudo-Anosov mapping class arising from Penner's construction using curves that intersect with the pattern of an odd cycle of length~$k\le l$.
In particular, Lemmas~\ref{bestflow} and~\ref{fromabove} imply~$\lambda(\phi)\ge\lambda(\varphi_k)\ge\lambda(\varphi_l)$.
\begin{table}[h]
\begin{tabular}{| c | c | c |}
\hline
cycle length & flow difference & dilatation \\ \hline
3 & 1 & $\approx 6.222$ \\ \hline
5 & 1 & $\approx 5.961$ \\ \hline
5 & 3 & $\approx 7.520$ \\ \hline
7 & 1 & $\approx 5.895$ \\ \hline
7 & 3 & $\approx 6.529$ \\ \hline
7 & 5 & $\approx 8.841$ \\ \hline
\end{tabular}
\smallskip
\caption{Some dilatations for short odd cycles.}
\label{cyclecasestable}
\end{table}

\emph{Case 3: the intersection graph only contains odd cycles of length~$k>l$:}
Take an odd cycle of minimal length~$k>l$ among odd cycles.
This cycle is necessarily an induced subgraph of the intersection graph.
Otherwise, the intersection graph would either have to contain a double edge (which we may rule out by Case~1) or an edge connecting two nonadjacent vertices of the cycle,
which implies the existence of an odd cycle of length~$<k$.
Hence, by Lemma~\ref{genusbound}, the genus of the surface~$N_{l+1}$ is bounded from below by~$k+1>l+1=g$,
a contradiction.
\end{proof}

\subsection{A proof of Theorem~\ref{evengenusexact} and almost a proof of Theorem~\ref{dilatationlimits}}
By Theorem~\ref{evengenus} and Proposition~\ref{specialisation},
we know that for even genus~$g$, the minimal dilatation~$\delta_P(N_g)$ equals the largest real solution of the equation~$$t-t^{\frac{g}{2g-2}}-t^{\frac{g-2}{2g-2}}-1=0.$$
Setting~$g=2k$ yields exactly the statement of Theorem~\ref{evengenusexact}.

We are now ready to show Theorem~\ref{dilatationlimits}, except for the existence of the limit~$\lim_{k\to\infty}\delta_P(N_{2k+1})$ for nonorientable closed surfaces of odd genus.

\begin{thm}
\label{dilatationlimitsalmost}
For the minimal dilatation~$\delta_P(N_g)$ among pseudo-Anosov mapping classes arising from Penner's construction on the nonorientable closed surface of genus~$g$,
the limit $\lim_{k\to\infty}\delta_P(N_{2k})$ exists, and
\begin{enumerate}
\item[(a)] $\lim_{k\to\infty}\delta_P(N_{2k}) = 3+2\sqrt{2}$,
\item[(b)] $\liminf_{k\to\infty}\delta_P(N_{2k+1})>3+2\sqrt{2}$.
\end{enumerate}
\end{thm}

\begin{proof}
By Theorem~\ref{evengenusexact},
we know that for even~$g$,~$\delta_P(N_g)$ equals the largest real solution of the equation~$$t-t^{\frac{g}{2g-2}}-t^{\frac{g-2}{2g-2}}-1=0.$$
As~$g\to\infty$, this solution converges to the largest real solution of the equation~$$t-2t^{\frac{1}{2}}-1=0,$$
which is~$3+2\sqrt{2}$, the square of the silver ratio. This proves the existence of the limit~$\lim_{k\to\infty}\delta_P(N_{2k})$ and the exact value in~(a).

In order to prove~(b), we show that a Penner dilatation on a nonorientable surface of odd genus is bounded from below by~$3+2\sqrt{2}+\delta$, where~$\delta=\frac{1}{10}$.
We can use similar steps as in the proof of Theorem~\ref{evengenus} and the values in Table~\ref{cyclecasestable}
to reduce the argument to the case where the intersection graph contains an induced cycle of length~$>9$.
By Corollary~\ref{fillingcyclecor}, the corresponding curves cannot fill the surface, since it is of odd genus.
Hence, the intersection graph must contain at least one more vertex connecting to the induced cycle.
In particular, it contains a subgraph of the form depicted in Figure~\ref{9subgraph} on the right.
In this case, we use Lemma~\ref{treebound} to obtain that the dilatation is bounded from below by~$5.946$.
\end{proof}

In order to show that the limit~$\lim_{k\to\infty}\delta_P(N_{2k+1})$ exists, we need a better grip on the actual minimal Penner dilatations~$\delta_P(N_g)$ for odd~$g$.
To this end, we study odd cycles with an extra vertex in the next section.

\section{Dilatation theory of the enriched cycle}
\label{enriched_section}
Let~$P_l$ be the \emph{enriched cycle of length}~$l$, that is, the~$l$-cycle with an additional vertex connecting to exactly one vertex of the cycle.
In order to deal with closed nonorientable surfaces of odd genus, we have to study these examples systematically.
The goal of this section is to prove the following analogues of Lemma~\ref{bestflow} and Lemma~\ref{fromabove} for enriched cycles.
For~$l$ odd, let~$\mu_l$ be the dilatation arising from Penner's construction using curves that have~$P_l$ as their intersection graph and a Dehn twist product with flow difference~$1$.
By the \emph{flow difference} of an enriched cycle we just mean the flow difference of the induced cycle obtained by removing the extra vertex.
As in Lemma~\ref{flowdiffconjugacy}, there is exactly one conjugacy class for each flow difference.

\begin{lem}
\label{panminimum}
The minimal dilatation arising from Penner's construction using curves that have~$P_l$ as their intersection graph is~$\mu_l$.
\end{lem}

\begin{lem}
\label{pandecreasing}
We have~$\mu_l\ge\mu_{l+2}$.
\end{lem}

In order to prove Lemma~\ref{panminimum} and Lemma~\ref{pandecreasing},
we will study fibred link representatives and the Perron-Frobenius eigenvectors of the matrices associated with the mapping classes arising from Penner's construction, respectively.

\subsection{Fibred links}
An oriented compact surface~$\Sigma$ (with oriented boundary) embedded in~$\Ss^3$ is a \emph{fibre surface}
if its interior~$\overset{\circ}{\Sigma}$ is the fibre of a locally-trivial fibre bundle~$p:\Ss^3\setminus\partial\Sigma\to\Ss^1$.
In this case, the oriented boundary~$\partial\Sigma$ is called a \emph{fibred link}.
Such a fibration is determined by a mapping class of~$\Sigma$ up to conjugation, the \emph{monodromy} of the fibration.

Given two oriented surfaces embedded in~$\Ss^3$, it is possible to obtain new oriented surfaces by \emph{plumbing}, that is, glueing the two surfaces
(which are separated by an oriented embedded sphere $\Ss^2$) together along a square (which is contained in the sphere~$\Ss^2$)
whose boundary arcs alternatingly belong to the boundary of one surface or the other,
see Figure~\ref{plumbingex} for an example. We assume that in the plumbing square, the orientations of both surfaces and the sphere agree.

\begin{figure}[h]
\begin{center}
\def\svgwidth{350pt}
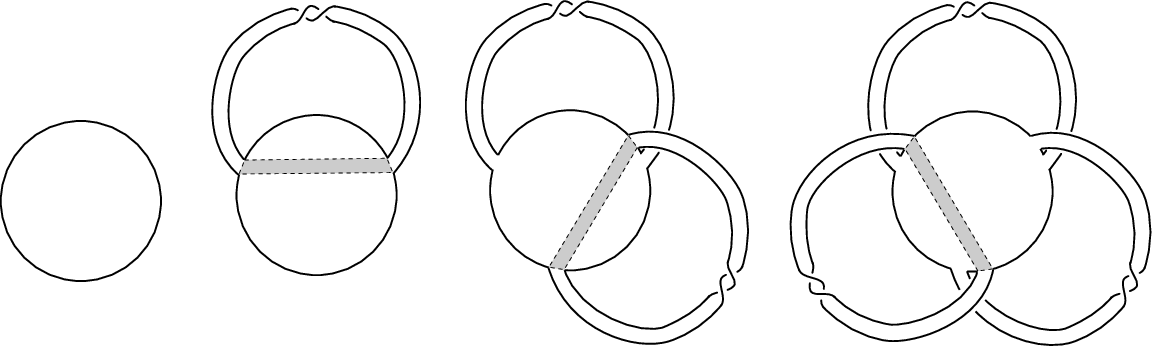
\caption{A fibre surface which is obtained by successive plumbing of Hopf bands to the standard disc. The plumbing square for each of the plumbings is coloured grey.}
\label{plumbingex}
\end{center}
\end{figure}

By a result of Stallings, a plumbing~$\Sigma$ of two fibre surfaces~$\Sigma_1$ and~$\Sigma_2$ is again a fibre surface~\cite{St}.
Furthermore, we now assume~$\Sigma_1$ to be on the negative side (the ``inside") of the sphere~$\Ss^2$ and~$\Sigma_2$ on the positive side (the ``outside") of the sphere~$\Ss^2$.
Then, the monodromy~$\phi$ of the plumbing is given by the composition~$\phi_1\circ\phi_2$ of the two monodromies~$\phi_1$ and~$\phi_2$ of the plumbing summands~$\Sigma_1$ and~$\Sigma_2$,
extended to~$\Sigma$ by the identity on~$\Sigma\setminus\Sigma_1$ and~$\Sigma\setminus\Sigma_2$, respectively.
For this to make sense, recall that mapping classes of surfaces with boundary are assumed to fix the boundary pointwise.

The positive Hopf band and the negative Hopf band are fibre surfaces and their monodromies are a positive Dehn twist and a negative Dehn twist along the core curve, respectively.
This fact, as well as Stallings' result, is accessibly explained by Baader and Graf, who interpret the concept of fibredness in terms of elastic cords~\cite{BaGr}.
By Stallings' result, a successive plumbing of Hopf bands yields a product of Dehn twists along the core curves of the Hopf bands plumbed.
In this way, it is possible to represent certain mapping classes as monodromies of fibred links.

\subsection{Realising Penner mapping classes as fibred link monodromies}
\label{rep_sec}
With the preceding discussion on fibred link monodromies, it is clear what we should do in order to obtain mapping classes that arise from Penner's construction as monodromies of fibred links:
plumb positive and negative Hopf bands such that the core curves of the positive Hopf bands do not intersect among themselves and likewise for the negative Hopf bands.
For instance, Figure~\ref{flowdiff2} depicts two fibre surfaces. Both are obtained from the closed standard disc, which is situated in the middle, by consecutive plumbing of Hopf bands.
There is a total of three positive and three negative Hopf bands plumbed in alternating fashion.
In this way, we obtain fibred links whose monodromy is a product of Dehn twists along curves which intersect each other with the pattern of a cycle.
Furthermore, the monodromy is a pseudo-Anosov mapping class arising from Penner's construction by using the core curves of the plumbed positive Hopf bands as one multicurve and the core curves of the plumbed negative Hopf bands as the other multicurve.
It is straightforward to see that we are able to represent any order of Dehn twists by varying the order of plumbing, and, in particular, every flow difference.

\begin{figure}[h]
\begin{center}
\def\svgwidth{360pt}
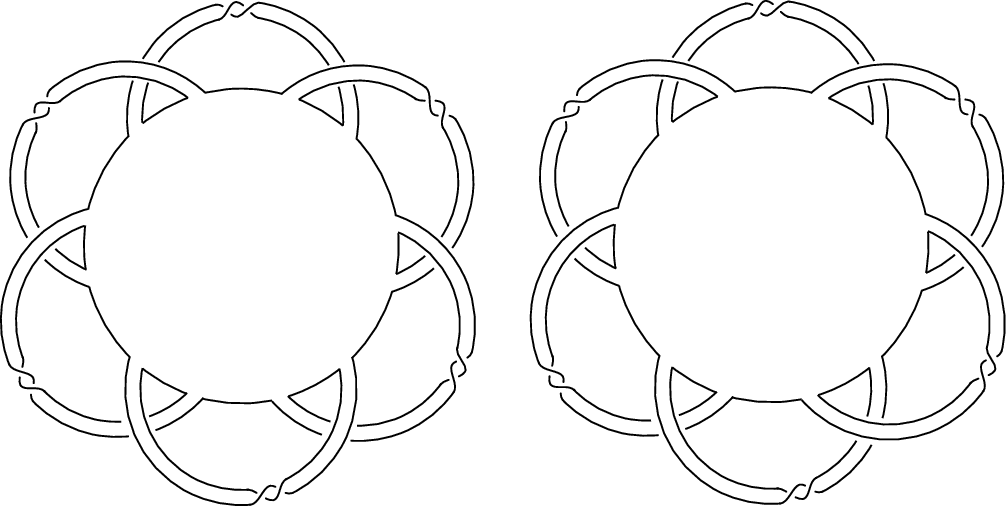
\caption{Fibred links realising flow difference~$0$ (on the left) and flow difference~$2$ (on the right).}
\label{flowdiff2}
\end{center}
\end{figure}

\begin{remark}\emph{
By the process of representing Penner mapping classes as monodromies of fibred links, we only obtain orientable surfaces.
However, we can still study the dilatations of Penner mapping classes on nonorientable surfaces by first lifting to the orientable double cover.
For example, if we want to study the dilatation of a Penner mapping class given by an odd cycle of length~$l$ with a flow difference~$d$,
we can lift it to a Penner mapping class on an orientable surface, with intersection graph the cycle of length~$2l$ and with flow difference~$2d$.
The dilatations of the original Penner mapping class and its lift agree and the latter can be represented by a fibred link monodromy.
In this context, we recall that a Penner dilatation depends only on the intersection graph and the twist order,
so it suffices to represent this information and not the actual Penner mapping classes.
}\end{remark}

\subsection{The Alexander polynomial of fibred links}
The Alexander polynomial~$\Delta_L$ of an oriented link~$L$ is defined recursively
by the skein relation~
\begin{align*}
\Delta_{L_+} = \Delta_{L_-}+\left(\sqrt{t}-\frac{1}{\sqrt{t}}\right)\Delta_{L_0},\tag{SR}\label{SR}
\end{align*}
and the initial condition~$\Delta_U = 1$, where~$U$ is the unknot.
For a fixed crossing, the links~$L_+$,~$L_-$ and~$L_0$ correspond to the positive version of the crossing, the negative version of the crossing and the orientation-preserving smoothing of the crossing, see Figure~\ref{skein}.
\begin{figure}[h]
\begin{center}
\def\svgwidth{200pt}
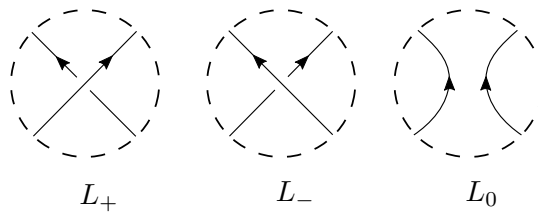
\caption{The links~$L_+$,~$L_-$ and~$L_0$ used in the skein relation are obtained by local adjustments at a crossing.}
\label{skein}
\end{center}
\end{figure}
The important property in our context is that for fibred links, the Alexander polynomial equals the characteristic polynomial of the action on the first homology of the fibre surface induced by the monodromy,
up to a normalisation factor, see, for example,~\cite{BuZi}. The normalisation factor equals~$\sqrt{t}^{\ b_1(L)}$ with a possible sign~$-1$ to make the leading coefficient positive.
This follows from the fact that the characteristic polynomial of a matrix of size~$b_1(L)$ is of degree~$b_1(L)$ and with leading coefficient~$+1$,
while the highest power of~$t$ appearing with nonzero coefficient in the Alexander polynomial of a fibred link is~$\sqrt{t}^{\ b_1(L)}$, where~$b_1(L)$ is the first Betti number of the fibre surface for~$L$.
We would like to stress that while the Alexander polynomial is often defined up to powers of the variable and up to sign, the skein-theoretic definition we use here gives a well-defined
Laurent polynomial in~$\sqrt{t}$.

\subsection{A proof of Lemma~\ref{panminimum}}

Let~$\Delta_{d,l}$ be the Alexander polynomial of the fibred link realisation of the mapping class arising via Penner's construction on curves that intersect with the pattern of a cycle of even length~$l=p+q$,
where every curve is twisted along exactly once and the twist order yields flow difference~$d=p-q$, where we may assume~$d\ge0$.
The following proposition contains the key result on Alexander polynomials of fibred link realisations.

\begin{prop}
\label{alexdiff}
In the above notation, we have
$$\Delta_{d,l}-\Delta_{d+2,l} =  \left(\sqrt{t}-\frac{1}{\sqrt{t}}\right)\left(\sqrt{t}^{\ d+1}-\sqrt{t}^{\ -(d+1)}\right).$$
\end{prop}

Our proof of Proposition~\ref{alexdiff} relies on the skein relation for the Alexander polynomial. We split it up into several separate statements.
For a natural number~$d$, let~$H_d$ be the link obtained by the closure of~$d$ stacked copies of the braid depicted on the left in Figure~\ref{H4}.
On the right in Figure~\ref{H4}, the link~$H_d$ is shown for~$d=4$.
\begin{figure}[h]
\begin{center}
\def\svgwidth{150pt}
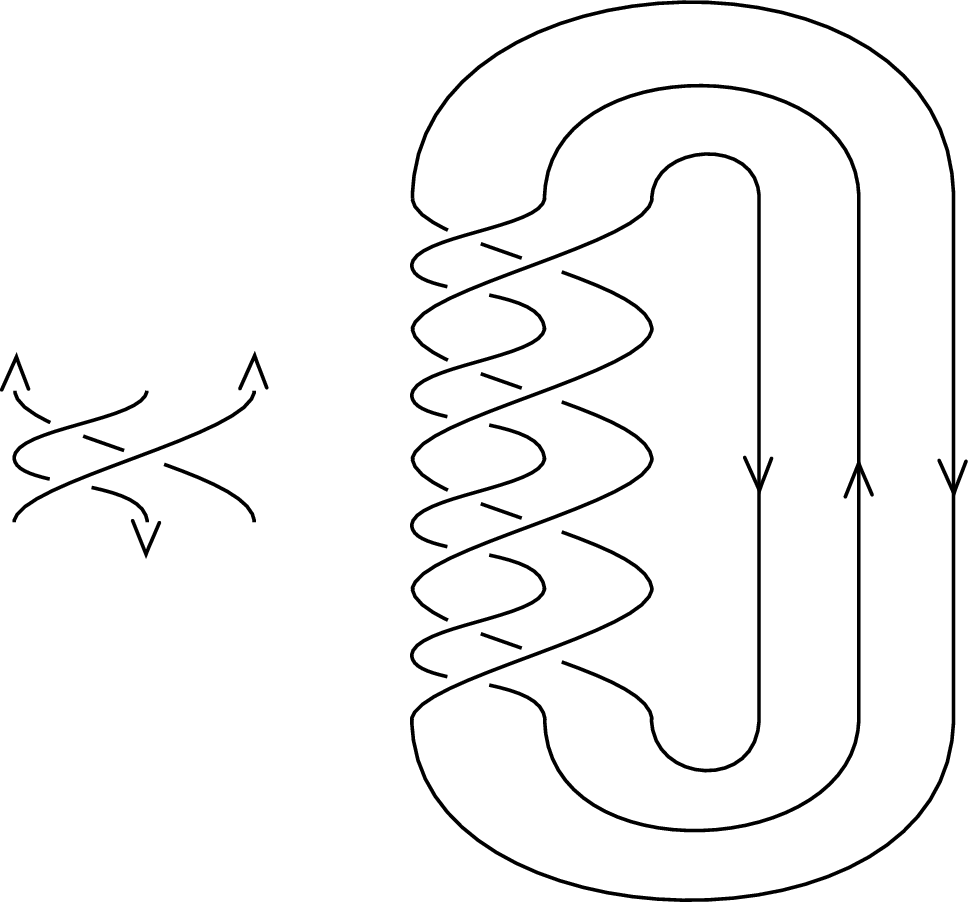
\caption{The braid building block for the links~$H_d$ (on the left) and the link~$H_4$ (on the right).}
\label{H4}
\end{center}
\end{figure}

\begin{lem}
\label{diffreduction}
Let~$p\ge q$,~$p'\ge q'$ be strictly positive natural numbers such that~$p+q=p'+q'=l$. Let~$d=p-q$ and~$d'=p'-q'$.
Then, we have
$$\Delta_{d,l}-\Delta_{d',l}=\Delta_{H_{d}}-\Delta_{H_{d'}}.$$
\end{lem}

\begin{proof}
Let~$L_{d}$ and~$L_{d'}$ be two links representing flow differences~$d$ and~$d'$, respectively, on a cycle of length~$l=p+q=p'=q'$,
consisting of~$l$ Hopf bands plumbed to a closed standard disc. We consider diagrams of~$L_{d}$ and~$L_{d'}$ as described in Section~\ref{rep_sec} and Figure~\ref{flowdiff2},
and apply the skein relation to a crossing of a twist of one of the plumbed bands. The change from a positive crossing to a negative crossing or vice-versa manifestly untwists the band.
The smoothing of the crossing as in the link~$L_0$ of the skein relation cuts the band. The resulting link is a plumbing of Hopf bands along a path.
In this case, whether a band passes over another one or vice-versa does not change the link up to isotopy. In particular, the~$L_0$-terms in the skein relation for~$L_{d}$ and~$L_{d'}$ agree.
Hence,~$\Delta_{d,l}-\Delta_{d',l}$ equals the difference of the Alexander polynomials of the links~$L_{d}$ and~$L_{d'}$, but with one Hopf band untwisted.
This argument can be repeated for each of the Hopf bands, which finally yields~$\Delta_{d,l}-\Delta_{d',l}=\Delta_{H_{d}}-\Delta_{H_{d'}}$,
since the link~$L_{d}$ with all Hopf bands untwisted is exactly the link~$H_{d}$.
Indeed, the link~$L_{d}$ with all Hopf bands untwisted can be divided into~$p$ sectors that resemble a positive half-twist on three strands and~$q$ sectors that resemble a negative half-twist on three strands,
with the orientation of the middle strand reversed, compare with Figure~\ref{sectorbraid}.
\begin{figure}[h]
\begin{center}
\def\svgwidth{360pt}
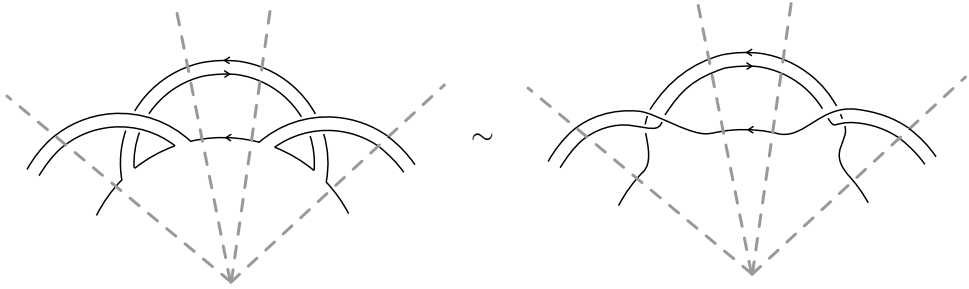
\caption{A link isotopy supported in bounded sectors.}
\label{sectorbraid}
\end{center}
\end{figure}
Since a positive and a negative half-twist cancel each other (which is very well perceivable in Figure~\ref{sectorbraid} on the right), we are left with~$d=p-q$ positive half-twists, that is, the link~$H_{d}$.
\end{proof}

\begin{lem}
\label{Hr}
For even~$d$, we have
$$\Delta_{H_d}=\left(\sqrt{t}-\frac{1}{\sqrt{t}}\right)\left(\Delta_{T_{2,d}}-2\sum_{i=1}^{\frac{d}{2}} \Delta_{T_{2,2i}}\right),$$
where~$T_{2,2i}$ denotes the~$(2,2i)$-torus link.
\end{lem}

\begin{proof}
The idea is to subsequently use the skein relation of the Alexander polynomial on all crossings where the middle strand of~$H_d$ passes below an other strand,
starting from the highest such crossing and proceeding to the lowest.
This allows for a computation after finitely many steps until the middle strand corresponds to a split component.
The crossing changes in the skein relation simplifies the linking of the middle strand with the other strands. This is depicted in Figure~\ref{Hrskein1}.
\begin{figure}[h]
\begin{center}
\def\svgwidth{240pt}
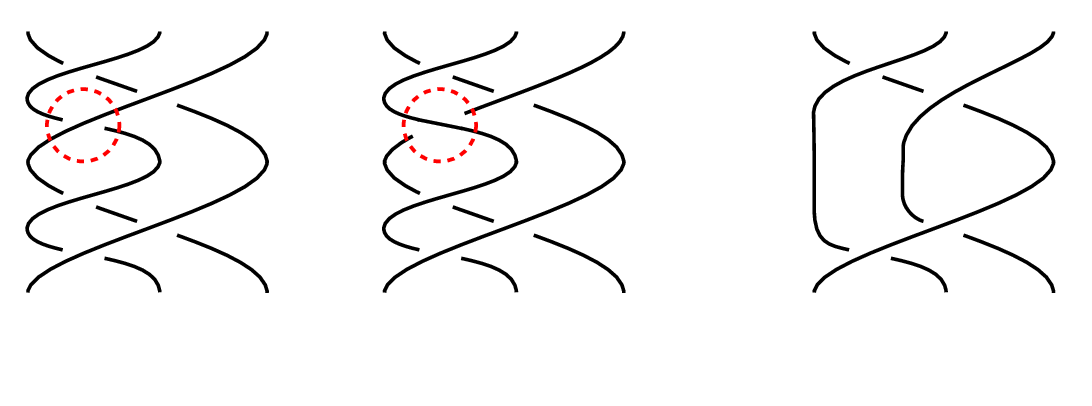
\caption{}
\label{Hrskein1}
\end{center}
\end{figure}
We now show that the~$L_0$-smoothings accumulate~$\Delta_{T_{2,2i}}$-summands with a coefficient~$-\left(\sqrt{t}-\frac{1}{\sqrt{t}}\right)$.
Assume that we have already changed the highest~$k\ge0$ crossings where the middle strand of~$H_d$ passes below an other strand.
We now describe what happens when we smooth the~$k+1$st crossing as in the~$L_0$-part of the skein relation.
Explicitly drawing the diagrams reveals that if~$k$ is even, we obtain a torus link~$T_{2,k}$ and if~$k$ is odd, we obtain a torus link~$T_{2,k+1}$.
Here, we also recall that~$d=p-q$ is even.
This is depicted in Figure~\ref{Hrskein2} for~$d=6$ and~$k=3$.
\begin{figure}[h]
\begin{center}
\def\svgwidth{250pt}
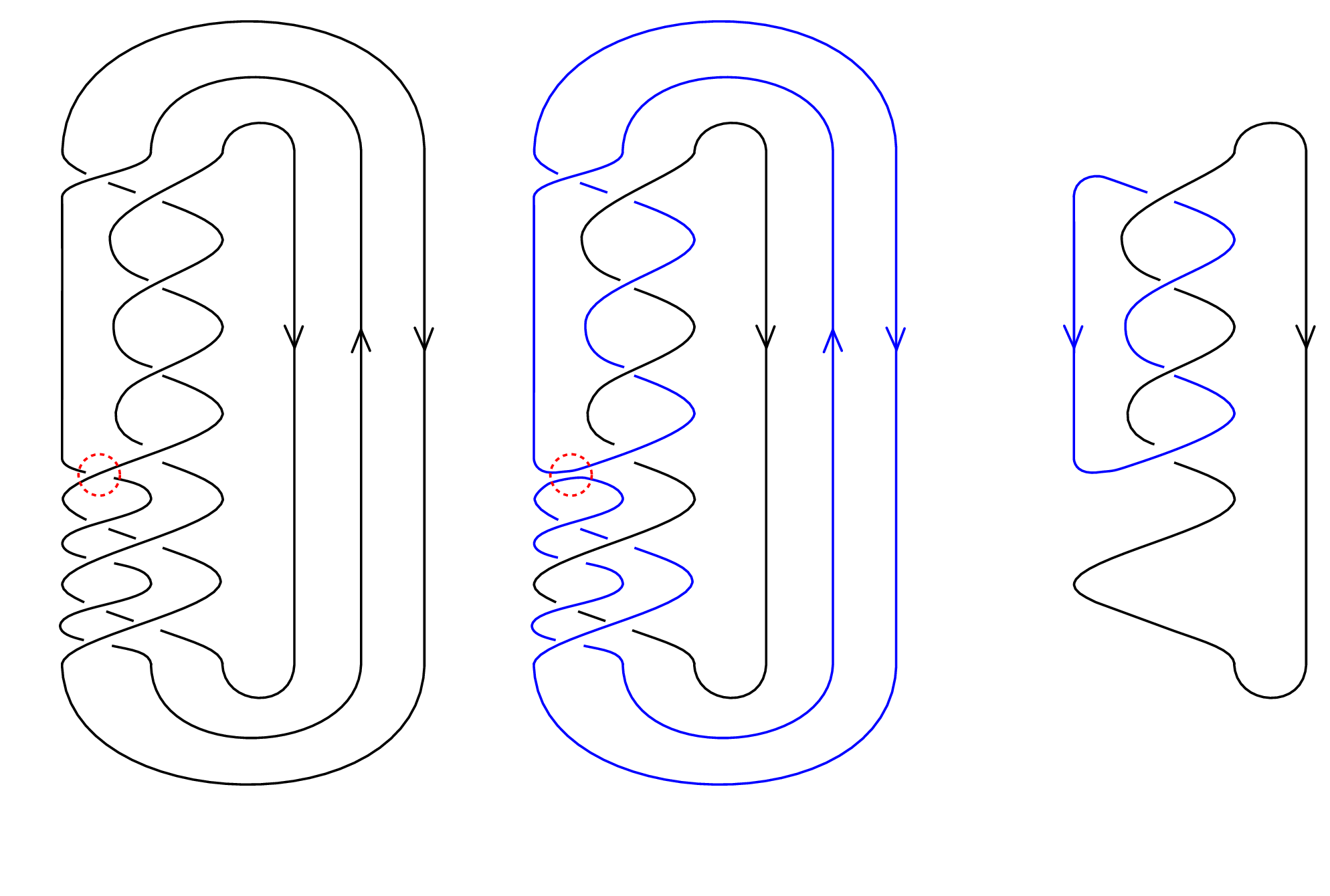
\caption{}
\label{Hrskein2}
\end{center}
\end{figure}
In the diagram for~$H_d$, there are~$d$ undercrossings of the middle strand. After we have changed all the crossings, the link is split and has Alexander polynomial~$0$.
Hence, the Alexander polynomial of~$H_d$ is the sum over all~$d$ undercrossings of the Alexander polynomial of the torus link obtained by the corresponding~$L_0$-smoothing,
as described above,
with a coefficient~$$-\left(\sqrt{t}-\frac{1}{\sqrt{t}}\right).$$
This yields the desired result.
\end{proof}

The following lemma is a standard fact on the Alexander polynomial of torus links. Using the skein relation for the Alexander polynomial, its verification is straightforward.

\begin{lem}
\label{torusalex}
$$\Delta_{T_{2,2i}}=\sum_{j=0}^{i-1}(-1)^{j}\left(\sqrt{t}^{\ 2i-1-2j}-\sqrt{t}^{\ -(2i-1-2j)}\right).$$
\end{lem}

\begin{proof}[Proof of Proposition~\ref{alexdiff}]
By Lemma~\ref{diffreduction}, we have
$$\Delta_{d,l}-\Delta_{d+2,l}=\Delta_{H_{d}}-\Delta_{H_{d+2}}.$$
On the other hand, using first Lemma~\ref{Hr} and then Lemma~\ref{torusalex}, we obtain
\begin{align*}
\Delta_{H_{d}}-\Delta_{H_{d+2}} &= \left(\sqrt{t}-\frac{1}{\sqrt{t}}\right)\left(\Delta_{T_{2,d+2}}+\Delta_{T_{2,d}}\right) \\
&= \left(\sqrt{t}-\frac{1}{\sqrt{t}}\right)\left(\sqrt{t}^{\ d+1}-\sqrt{t}^{\ -(d+1)}\right),
\end{align*}
which is what we wanted to show.
\end{proof}

\begin{lem}
\label{leadingcoeff}
The leading coefficient of~$\Delta_{d,l}$ is~$+1$.
\end{lem}

\begin{proof}
Using the exact same skein relations as in the proof of Lemma~\ref{diffreduction},
we have that~$\Delta_{d,l}$ is a sum of~$\Delta_{H_{d}}$ and summands of Alexander polynomials of plumbings of Hopf bands of alternating kind along a path, with varying coefficients from the skein relation.
A careful inspection reveals that the leading coefficient of~$\Delta_{d,l}$ equals the leading coefficient of the Alexander polynomial of the longest such path starting and ending with a negative Hopf band.
With a recursion on the length of such a path, one can show that this leading coefficient is~$+1$.
\end{proof}

We will show that for the enriched cycle, the dilatation is a monotonic function of the absolute value of the flow difference.
In particular, this implies that for an odd enriched cycle of length~$l$,
the minimal Penner dilatation is obtained by the example with flow difference~$1$.

\begin{proof}[Proof of Lemma~\ref{panminimum}]
We will compare the dilatations of Penner mapping classes obtained by twisting along curves which intersect like an enriched odd cycle, when we vary the flow difference associated with the order of twisting.
A concrete surface (of genus~$7$, that is,~$l=5$) for which we can build such an example is shown in Figure~\ref{nonorientable_examples} on the right.
We first lift the mapping class to the double cover orienting the surface. By doing this, we double the length of the cycle and the flow difference.
Furthermore, there are now two extra vertices connecting to the cycle at opposite ends, one corresponding to a curve along which we twist positively
and one corresponding to a curve along which we twist negatively. This lifted mapping class has the same dilatation as the Penner mapping class we started with.

We find a fibred link representative of the lifted mapping class by taking the usual representative for the cycle and plumbing two extra Hopf bands~$H_1$ and~$H_2$,
such as in Figure~\ref{enrichedskein} in the top middle.
\begin{figure}[h]
\begin{center}
\def\svgwidth{300pt}
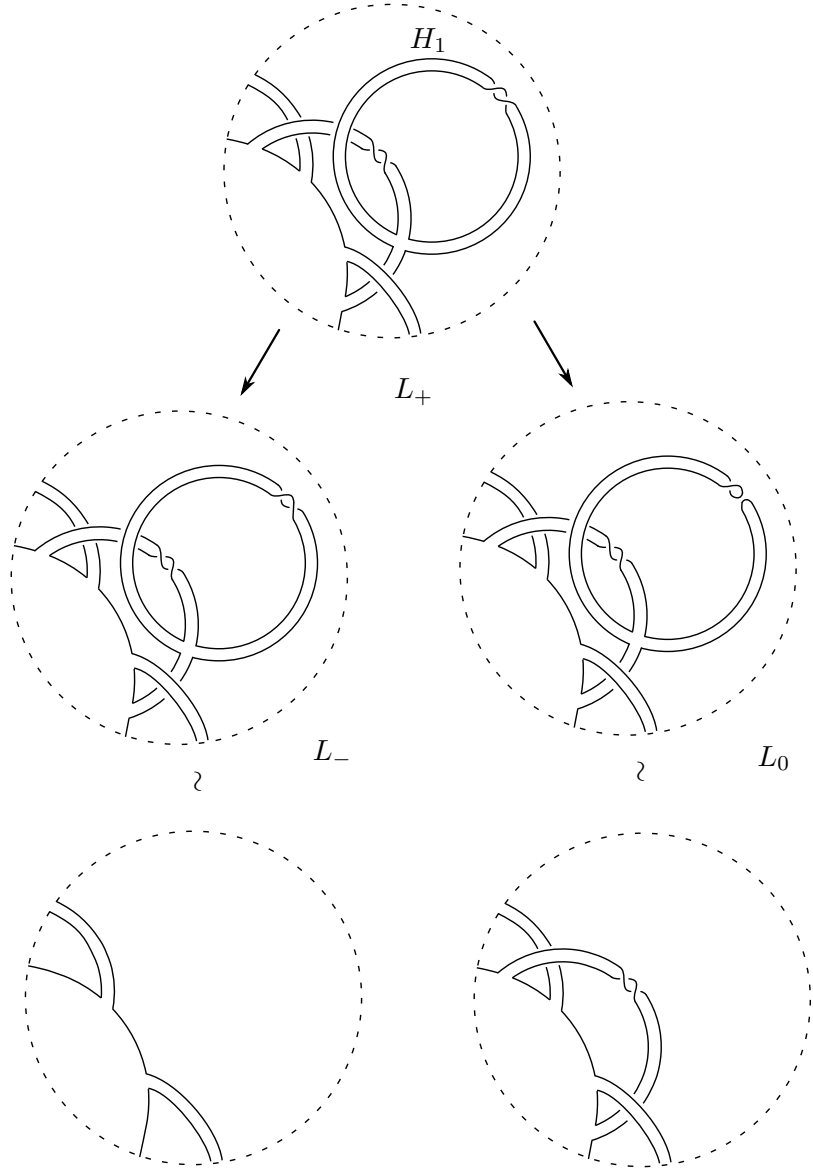
\caption{The links appearing in the skein relation~\eqref{SR} for the fibred link representations of enriched cycles, compare also with Figure~\ref{skein}.
The crossing used in the skein relation is one of the crossings of the band~$H_1$.}
\label{enrichedskein}
\end{center}
\end{figure}
We assume~$H_1$ to be a positive Hopf band and~$H_2$ to be a negative Hopf band.
Let~$L(d,l)$ be the enriched fibred link representative, where~$l$ is the even length of the cycle, and~$d\ge0$ is the flow difference.

We use the skein relation~\eqref{SR} for the enriched fibred link representative~$L(d,l)=L(d,l)_+$ at a positive crossing of~$H_1$. This yields
$$\Delta_{L(d,l)} = \Delta_{L(d,l)_-} + \left(\sqrt{t}-\frac{1}{\sqrt{t}}\right)\Delta_{L(d,l)_0}.$$

The links~$L_+ = {L(d,l)} $,~$L_- ={L(d,l)_-} $ and~$L_0={L(d,l)_0} $ appearing in the skein relation are depicted in Figure~\ref{enrichedskein}.
We note that~$L_-$ is given by a plumbing of Hopf bands along a tree. In particular, the Alexander polynomial~$\Delta_{L_-}$ does not depend on the order of twisting,
and, in particular, does not depend on the flow difference.
From this, we deduce~$$\Delta_{L(d,l)} - \Delta_{L(d+2,l)} = \left(\sqrt{t}-\frac{1}{\sqrt{t}}\right)\left(\Delta_{L(d,l)_0} -\Delta_{L(d+2,l)_0} \right).$$

The links~$L=L(d,l)_0$ and~$L'=L(d+2,l)_0$ are fibred link representatives of enriched cycles of even length~$l$ and flow difference~$d\ge0$ and~$d+2$, respectively.
Furthermore, the extra band~$H_2$ corresponding to the extra vertex which is negative.
We can get rid of the band~$H_2$ by another use of the skein relation, similarly to the skein relation we used to get rid of~$H_1$,
but we have to take care of the change in sign of the band.
Using the skein relation on a negative crossing of~$H_2$, which is negative, we obtain
$$\Delta_{L} = \Delta_{L_+} - \left(\sqrt{t}-\frac{1}{\sqrt{t}}\right)\Delta_{L_0}$$
and
$$\Delta_{L'} = \Delta_{L'_+} - \left(\sqrt{t}-\frac{1}{\sqrt{t}}\right)\Delta_{L'_0},$$
respectively.
We note that again, the links~$L_+$ and~$L'_+$ are given by a plumbing of Hopf bands along a forest. In particular, we have~$\Delta_{L_+}=\Delta_{L'_+}$,
and this yields
$$\Delta_{L}-\Delta_{L'} = -\left(\sqrt{t}-\frac{1}{\sqrt{t}}\right)\left(\Delta_{L_0}-\Delta_{L'_0}\right),$$
where~$\Delta_{L_0}$ equals~$\Delta_{d,l}$ and~$\Delta_{L'_0}$ equals~$\Delta_{d+2,l}$.
Applying Proposition~\ref{alexdiff}, this gives
\begin{align*}
\Delta_{L(l,d)} - \Delta_{L(l,d+2)} &= -\left(\sqrt{t}-\frac{1}{\sqrt{t}}\right)^2\left(\Delta_{L_0}-\Delta_{L'_0}\right)\\
&= -\left(\sqrt{t}-\frac{1}{\sqrt{t}}\right)^3\left(\sqrt{t}^{\ d+1}-\sqrt{t}^{\ -(d+1)}\right).
\end{align*}

Now, let~$\chi_{d,l}$ and~$\chi_{d+2,l}$ be the characteristic polynomials of the action induced on the first homology by the monodromies of
the fibred link representatives~$L(d+2,l)$ and~$L(d+2,l)$, respectively.

We note that the leading coefficient of both~$\Delta_{L(d,l)}$ and~$\Delta_{L(d+2,l)}$ is~$-1$.
Indeed, we have used the skein relation on one positive and one negative crossing to go from~$\Delta_{L(d,l)}$ and~$\Delta_{L(d+2,l)}$
to~$\Delta_{d,l}$ and~$\Delta_{d+2,l}$, respectively, which have leading coefficient~$+1$. Tracking the sign of the leading coefficient
through the two skein relations yields that it switches. In particular, the leading coefficients of~$\Delta_{L(d,l)}$ and~$\Delta_{L(d+2,l)}$ are~$-1$,
since the leading coefficients of~$\Delta_{d,l}$ and~$\Delta_{d+2,l}$ are~$+1$ by Lemma~\ref{leadingcoeff}.
This means that to normalise the Alexander polynomials~$\Delta_{L(d,l)}$ and~$\Delta_{L(d+2,l)}$ to the characteristic polynomials~$\chi_{d,l}$ and~$\chi_{d+2,l}$
of the action induced on first homology by the monodromies of~$L(d,l)$ and~$L(d+2,l)$, respectively,
we have to multiply by~$-(\sqrt{t})^{l+2}$. This yields a difference of
\begin{align*}
\chi_{d,l} -\chi_{d+2,l} &=  \left(\sqrt{t}\right)^{l+2}\left(\sqrt{t}-\frac{1}{\sqrt{t}}\right)^3\left(\sqrt{t}^{\ d+1}-\sqrt{t}^{\ -(d+1)}\right)\\
&= \left(t-1\right)^3\left(t^{\frac{l+d}{2}}-t^\frac{l-d-2}{2}\right).
\end{align*}
Clearly, this difference is strictly positive for any real number~$t>1$.
In particular, when evaluated at real numbers strictly greater than~$1$,
the characteristic polynomial of the action corresponding to the flow difference~$d\ge0$ is strictly greater than the characteristic polynomial of the action corresponding to the flow difference~$d+2$.
This implies that the largest real root is strictly greater for the characteristic polynomial of the action corresponding to the flow difference~$d+2$, and thus proves the claim.
\end{proof}

\subsection{A monotonicity criterion and a proof of Lemma~\ref{pandecreasing}}
Let~$\phi$ be a mapping class obtained from Penner's construction using curves with intersection graph~$\Gamma$, and such that every curve gets twisted along exactly once.
We further assume~$\Gamma$ to have only simple edges.
The intersection graph~$\Gamma$ is equipped with an acyclic orientation given by the order in which the curves used in Penner's construction get twisted along.
In this context, the matrix~$\rho(\phi)$ associated with~$\phi$ in Penner's construction (Theorem~\ref{pennerconstruction_thm})
only depends on the intersection graph~$\Gamma$ and the acyclic orientation.
In particular, the same is true for the dilatation, and we consider both~$\rho(\phi)$ and~$\lambda(\phi)$
as a function of the intersection graph~$\Gamma$ with its acyclic orientation. We write~$\rho(\Gamma)$ and~$\lambda(\Gamma)$, respectively.

Let~$y\in\R^n$ with coefficients~$y_i$ (corresponding to the vertices~$v_i$ of~$\Gamma$)
be a Perron-Frobenius eigenvector of the matrix~$\rho(\Gamma)$.
This is a vector whose entries satisfy the following set of equations.
For each vertex~$v_i$ of~$\Gamma$, let~$v_{o_1},\dots,v_{o_r}$ be the vertices of~$\Gamma$ connected to~$v_i$ by an edge pointing away from~$v_i$.
Similarly, let~$v_{i_1},\dots,v_{i_q}$ be the vertices of~$\Gamma$ connected to~$v_i$ by an edge pointing towards~$v_i$.
Then the matrix~$\rho(\Gamma)$ defined in Penner's construction acts on an arbitrary vector~$x\in\R^n$ as follows:
\begin{align*}
\tag{P}\label{P}
\left(\rho(\phi)x\right)_i =\ &x_i \\ &+ x_{o_1}+\cdots+x_{o_r}\\&+\left(\rho(\phi)x\right)_{i_1}+\cdots+\left(\rho(\phi)x\right)_{i_q}.
\end{align*}

In other words, to a weight we add all the weights adjacent in the graph, and we do this to all the weights in the order given by the acyclic orientation.
In particular, the Perron-Frobenius eigenvector~$y$ of~$\rho(\Gamma)$ satisfies the equation
\begin{align*}
\tag{PF}\label{PF}
\lambda y_i =\ &y_i\\ &+y_{o_1}+\cdots+y_{o_r}\\ & + \lambda(y_{i_1}+\cdots+y_{i_q}).
\end{align*}

\begin{prop}
\label{localdecrease_prop}
Assume that locally around a vertex~$v_i$ the acyclic orientation of~$\Gamma$
looks like a source-sink path as in Figure~\ref{localdecrease} on the left.
Assume furthermore that the coefficient~$y_i$ of the Perron-Frobenius eigenvector~$y$ corresponding to~$v_i$ is smaller than
or equal to the coefficients~$y_{i-2}$ and~$y_{i+2}$ corresponding to the two outer vertices.
Then, locally prolonging the path by two vertices as shown in Figure~\ref{localdecrease} on the right yields an acyclically oriented graph~$\Gamma'$ with associated
Penner dilatation~$\lambda'(\Gamma)$ such that~$\lambda'(\Gamma)\le\lambda(\Gamma)$.
\begin{figure}[H]
\begin{center}
\def\svgwidth{230pt}
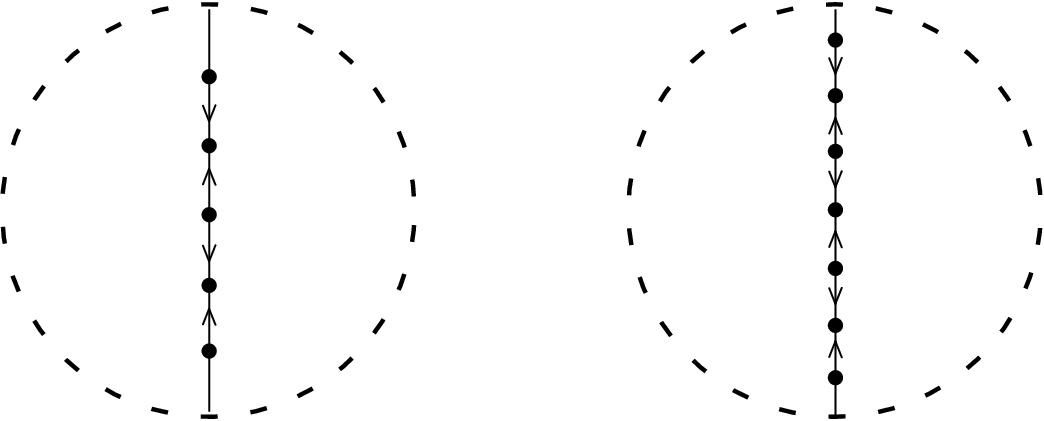
\caption{}
\label{localdecrease}
\end{center}
\end{figure}
\end{prop}

The following lemma is a standard description of the Perron-Frobenius eigenvalue of a Perron-Frobenius matrix.

\begin{lem}
\label{minmax}
Let~$A$ be a Perron-Frobenius matrix of size~$n\times n$, and let~$\lambda$ be its Perron-Frobenius eigenvalue.
Then $$\lambda=\mathrm{min}\left(\underset{x_i\ne0}{\mathrm{max}}\left(\frac{(Ax)_i}{x_i}\right)\right),$$
where the minimum is taken over all nonnegative vectors~$x\in\R^n\setminus\{0\}$.
\end{lem}

\begin{proof}[Proof of Proposition~\ref{localdecrease_prop}]

Let~$\rho(\Gamma')$ be the matrix of size~$(n+2)\times(n+2)$ associated with the acyclically oriented graph~$\Gamma'$ by Penner's construction.
We will describe a nonnegative vector~$x\in\R^{n+2}$ such that for each coefficient~$x_i$, we have $$\frac{(\rho(\Gamma')x)_i}{x_i}\le\lambda(\Gamma).$$
It then follows from Lemma~\ref{minmax} that the Perron-Frobenius eigenvalue~$\lambda(\Gamma')$ associated with~$\Gamma'$ is bounded from above by~$\lambda(\Gamma)$.

Let~$x\in\R^{n+2}$ be the vector with entries as shown in Figure~\ref{localmove} on the right, where we assume
that all the entries corresponding to vertices outside the local picture are equal to the corresponding entry~$y_j$ of the Perron-Frobenius eigenvector~$y$ of~$\rho(\Gamma)$.

\begin{figure}[h]
\begin{center}
\def\svgwidth{300pt}
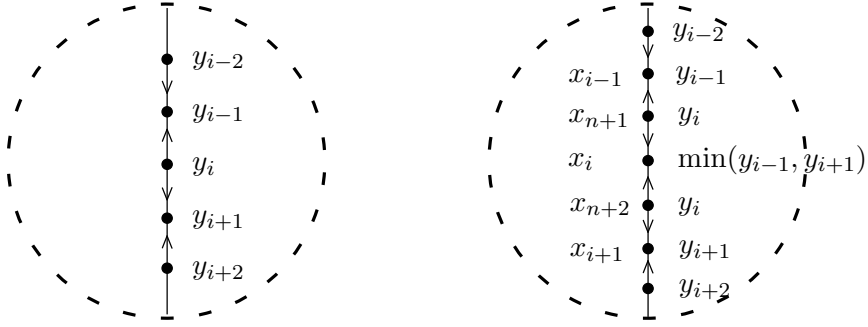
\caption{The graphs~$\Gamma$ on the left and the graph~$\Gamma'$ on the right, obtained by a local modification around the vertex~$v_i$.}
\label{localmove}
\end{center}
\end{figure}

Except for the five middle vertices, all the entries corresponding to the vertices on the right satisfy the exact same equations~\eqref{P}, and thus satisfy also the respective equations~\eqref{PF}.
In particular, the corresponding entries~$x_j$ satisfy $$\frac{(\rho(\Gamma')x)_j}{x_j}=\lambda.$$
We still have to show
\begin{align*}
\tag{$\ast$}\label{ast}
\frac{(\rho(\Gamma')x)_j}{x_j}\le\lambda
\end{align*}
for the entries~$x_{i-1}, x_i, x_{i+1}, x_{n+1}$ and~$x_{n+2}$ corresponding to the five middle vertices.
Using the equations~\eqref{P} and~\eqref{PF} as well as the assumption~$$y_i\le\mathrm{min}(y_{i-2},y_{i+2}),$$ we now verify this by direct computation.
We first note that
\begin{align*}
(\rho(\Gamma')x)_{n+1} &= x_{n+1} + x_{i-1}+x_i \\
&=y_i + y_{i-1} + \mathrm{min}(y_{i-1},y_{i+1})\\
&\le \lambda(\Gamma)y_i=\lambda(\Gamma)x_{n+1},
\end{align*}
which proves~\eqref{ast} for~$x_{n+1}$. The analogue computation for~$x_{n+2}$ also yields $$(\rho(\Gamma')x)_{n+2}\le \lambda(\Gamma)x_{n+2}.$$
Similarly, for~$x_{i-1}$, we have
\begin{align*}
(\rho(\Gamma')x)_{i-1} &= x_{i-1} + (\rho(\Gamma')x)_{n+1}+(\rho(\Gamma')x)_{i-2} \\
&= y_{i-1} + (\rho(\Gamma')x)_{n+1} + (\rho(\Gamma')x)_{i-2} \\
&\le y_{i-1} + \lambda(\Gamma)y_i + \lambda(\Gamma)y_{i-2} \\
&=\lambda(\Gamma)y_{i-1}=\lambda(\Gamma)x_{i-1},
\end{align*}
which proves~\eqref{ast} for~$x_{i-1}$. The analogue computation for~$x_{i+1}$ yields $$(\rho(\Gamma')x)_{i+1}\le \lambda(\Gamma)x_{i+1}.$$
Finally, for~$x_i$, we have
\begin{align*}
(\rho(\Gamma')x)_{i} &= x_{i} + (\rho(\Gamma')x)_{n+1}+(\rho(\Gamma')x)_{n+2} \\
&\le \mathrm{min}(y_{i-1},y_{i+1}) + 2\lambda(\Gamma)y_i \\
&\le \mathrm{min}(y_{i-1},y_{i+1}) + \lambda(\Gamma)y_i + \lambda(\Gamma)\mathrm{min}(y_{i-2},y_{i+2})\\
&\le \mathrm{min}(\lambda(\Gamma)y_{i-1}, \lambda(\Gamma)y_{i+1}) \\
&=\lambda(\Gamma)\mathrm{min}(y_{i-1},y_{i+1})=\lambda(\Gamma)x_i,
\end{align*}
which proves~\eqref{ast} for~$x_i$ and finishes the proof.
\end{proof}

We are now ready to show that the sequence of dilatations~$(\mu_l)$ of the Penner mapping classes associated with the enriched cycle of odd length~$l$ and flow difference~$1$ is
nowhere increasing.

\begin{proof}[Proof of Lemma~\ref{pandecreasing}]
For~$l=3,5,7,9,11,13$ we simply check the statement by hand (on a computer) and notice that~$\mu_{13}<6.13$.
The results of the calculation are given in Table~\ref{enrichedcyclecasestable}.
\begin{table}[h]
\begin{tabular}{| c | c |}
\hline
$l$ & $\mu_l$\\ \hline
3 & $\approx 6.996$ \\ \hline
5 & $\approx 6.452$ \\ \hline
7 & $\approx 6.277$ \\ \hline
9 & $\approx 6.194$ \\ \hline
11 & $\approx 6.148$ \\ \hline
13 & $\approx 6.120$ \\ \hline
\end{tabular}
\smallskip
\caption{Some values of~$\mu_l$.}
\label{enrichedcyclecasestable}
\end{table}
We now proceed by induction on the length~$l$ of the cycle.
Assume we have shown the statement up to cycles of length~$\le 2n-1$.
We want to show~$\mu_{2n-1}\ge\mu_{2n+1}$.
The idea is to apply Proposition~\ref{localdecrease_prop} to the sink whose corresponding entry of the Perron-Frobenius eigenvector is minimal.

Let~$P_{2n-1}$ be the enriched cycle of length~$2n-1$, acyclically oriented as in Figure~\ref{enriched}, where the edges which are not displayed are oriented in alternating fashion.
Clearly, the absolute value of the flow difference equals~$1$, so the associated Penner dilatation~$\Gamma(P_{2n-1})$ is~$\mu_{2n-1}$.
Let~$y\in\R^{2n}$ be the Perron-Frobenius eigenvector for the matrix~$\rho(P_{2n-1})$, where its~$i$th entry~$y_i$ corresponds to the vertex~$v_i$ of~$P_{2n-1}$.
\begin{figure}[h]
\begin{center}
\def\svgwidth{180pt}
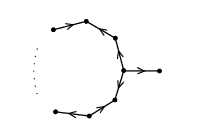
\caption{}
\label{enriched}
\end{center}
\end{figure}
We are interested in the minimal entry of~$y$. It is a direct observation that if there is an edge pointing from~$v_j$ to~$v_k$, then~$y_j\le y_k$.
This follows from the equation~\eqref{PF}.
In particular, the minimal entry of~$y$ corresponds to a source.

If the minimal entry~$y_i$ of~$y$ corresponds to~$v_i\ne v_1$, then we can apply Proposition~\ref{localdecrease_prop}.
This yields a Penner mapping class on the enriched cycle of length~$2n+1$ with the same flow difference and smaller dilatation.
In particular, we have~$\mu_{2n-1}\ge\mu_{2n+1}$ and we are done.

Now assume the minimal entry of~$y$ is~$y_1$. Consider Figure~\ref{H},
\begin{figure}[h]
\begin{center}
\def\svgwidth{180pt}
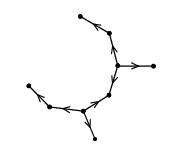
\caption{}
\label{H}
\end{center}
\end{figure}
which describes an acyclically oriented tree~$\Gamma$ and thus a Penner dilatation~$\lambda(\Gamma)$ and an associated matrix~$\rho(\Gamma)$.
Furthermore, Figure~\ref{H} describes a nonnegative vector~$x\in\R^9$, with entries given by the indicated vertex weights.
Using the equations~\eqref{P} and~\eqref{PF} and the assumption that~$y_1$ is minimal among the entries of~$y$,
one can show that, for any~$1\le j\le 9$,
\begin{align*}\tag{$\ast\ast$}\label{astast}
\frac{(\rho(\Gamma)x)_j}{x_j}\le\mu_{2n-1}<6.13.
\end{align*}
This can be verified very similar to the calculations in the proof of Lemma~\ref{pandecreasing}.
For all entries except the one with weight~$y_{2n-1}$, the inequality~\eqref{astast} follows very directly from a comparison with the
equation~\eqref{PF} for the corresponding entry of the Perron-Frobenius eigenvector~$y$.
To show~\eqref{astast} for the entry with weight~$y_{2n-1}$, it is necessary to make use of the assumption that~$y_1$ is smaller than~$y_{2n-2}$.

Proving~\eqref{astast} for all entries~$x_j$ of the vector~$x$ yields a contradiction,
since the spectral radius of~$\rho(\Gamma)$ can be calculated directly and is strictly larger than~$6.13$.
\end{proof}

\section{Odd genus minimal dilatations}
\label{oddgenus_section}

Let~$l$ be an odd natural number and let~$\psi_l$ be the mapping class arising from Penner's construction
using curves with an enriched~$l$-cycle~$P_l$ as their intersection graph, and with flow difference~$1$.
The dilatation of~$\psi_l$ equals~$\mu_l$. We will use the values of~$\mu_l$ calculated in Table~\ref{enrichedcyclecasestable}.
We are now ready to show that the mapping classes~$\psi_l$ minimise the dilatation among mapping classes
arising from Penner's construction on nonorientable surfaces of odd genus.
Note that we have to show this for genus greater than or equal to~$5$, since the genus~$3$ nonorientable closed surface does not admit pseudo-Anosov mapping classes.

\begin{thm}
\label{oddgenus}
The mapping class~$\psi_l$ minimises the dilatation among mapping classes arising from Penner's construction on the nonorientable closed surface of odd genus~$g=l+2$.
\end{thm}

\begin{proof}
Let $N_{l+2}$ be the nonorientable closed surface of even genus~$l+2$. We know that there exists the mapping class~$\psi_l$ on~$N_{l+2}$, with dilatation~$\mu_l=\lambda(\psi_l)$.
Let~$\phi$ be any mapping class on~$N_{l+2}$ arising from Penner's construction, where we assume that every curve used for the construction of~$\phi$ gets twisted along exactly once.
Exactly as in the proof of Theorem~\ref{evengenus}, we distinguish cases depending on the intersection graph of the curves used in the construction of~$\phi$.

\emph{Case 1: the intersection graph contains a double edge.} We use the same argument as in Case~1 of Theorem~\ref{evengenus}.
Let~$c_1$ and~$c_2$ be two curves that intersect at least twice. There must be at least one other curve~$c_3$ intersecting either~$c_1$ or~$c_2$,
and the intersection graph of the curves~$\{c_i\}$ contains the tree~$\Gamma$ with three vertices, one double edge and one simple edge as a subgraph,
depicted in Figure~\ref{9subgraph} on the left. As in the proof of Theorem~\ref{evengenus}, we obtain $\lambda(\phi)\ge\frac{7+3\sqrt{5}}{2}\approx6.854.$
Note that~$\mu_5\approx 6.452$ but~$\mu_3\approx 6.996$,
so the argument works for genus at least~$7$. In order to accommodate genus~$5$ in the argument, we need to consider also slightly larger subgraphs than the one used in Case~$1$ of the proof of Theorem~\ref{evengenus}: there must be at least one other edge, since the intersection graph must contain an odd cycle by Lemma~\ref{notbipartite}.
More precisely, the intersection graph actually contains one one the four graphs shown in Figure~\ref{double} as a subgraph.
\begin{figure}[h]
\begin{center}
\def\svgwidth{110pt}
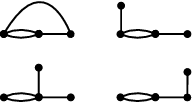
\caption{}
\label{double}
\end{center}
\end{figure}
All Penner mapping classes associated with one of those four graphs can directly be shown to have dilatation bounded from below by~$7$.
Hence, by monotonicity of the spectral radius of nonnegative matrices under~$``\le"$, we can exclude double edges of the intersection graph also for genus~$5$.

\emph{Case 2: the intersection graph contains an odd cycle of length~$k\le l$}: We can use the same argument as in the proof of Theorem~\ref{evengenus}, this time invoking the monotonicity lemmas we have
proved for enriched cycles. Since an odd cycle cannot fill a nonorientable surface of odd genus by Corollary~\ref{fillingcyclecor},
we may assume that the intersection graph contains an enriched cycle of length~$k\le l$
as an induced subgraph. In particular, the dilatation~$\lambda(\phi)$ is bounded from below by the dilatation of a pseudo-Anosov mapping class arising from Penner's construction using curves that intersect with the pattern of an enriched odd cycle of length~$k\le l$. In particular, Lemmas~\ref{panminimum} and~\ref{pandecreasing} directly imply~$\lambda(\phi)\ge\lambda(\psi_k)\ge\lambda(\psi_l)$.

\emph{Case 3: the intersection graph only contains odd cycles of length~$k>l$:}
Take an odd cycle of minimal length~$k>l$ among odd cycles.
Exactly as in Case~3 of the proof of Theorem~\ref{evengenus}, we may assume this cycle is in fact an induced subgraph of the intersection graph.
Hence, by Lemma~\ref{genusbound}, the genus of the surface is bounded from below by~$k+1>l+2=g$,
a contradiction.
\end{proof}

We can now complete the proof of Theorem~\ref{dilatationlimits}.

\begin{proof}[Proof of Theorem~\ref{dilatationlimits}]
All statements of Theorem~\ref{dilatationlimits} are implied by Theorem~\ref{dilatationlimitsalmost},
except for the existence of the limit~$\lim_{k\to\infty}\delta_P(N_{2k+1})$.
To show that this limit indeed exists, the only thing we have to note is that the sequence~$\delta_P(N_{2k+1})$ is not increasing in~$k$.
But this is a direct consequence of Lemma~\ref{pandecreasing}, since~$\delta_P(N_{2k+1})=\lambda(\psi_{2k-1})=\mu_{2k-1}$ by Theorem~\ref{oddgenus}.
\end{proof}

\end{document}